\documentclass{amsart}
\usepackage{latexsym, amsmath, amssymb, amsthm, mathrsfs}
\usepackage[dvipsnames]{xcolor}
\usepackage{graphicx}
\usepackage[colorlinks=true,linkcolor=blue,urlcolor=blue, citecolor=blue]%
  {hyperref}
  
\usepackage[
backend=biber,
style=alphabetic
]{biblatex}

\usepackage{tikz-cd}
\usepackage{caption}
\usepackage{subcaption}
\usepackage{nicematrix}
\NiceMatrixOptions{code-for-first-col = \color{red},code-for-first-row = \color{red}}

\usetikzlibrary{patterns}
\usepackage{booktabs}
\usepackage{multirow}
\usepackage{longtable}
\usepackage{changepage}

\newcommand{\tabitem}{~~\llap{\textbullet}~~}

\newtheorem{theorem}{Theorem}[section]
\newtheorem{lemma}[theorem]{Lemma}
\newtheorem{proposition}[theorem]{Proposition}
\newtheorem{corollary}[theorem]{Corollary}

\theoremstyle{definition}
\newtheorem{definition}[theorem]{Definition}

\theoremstyle{remark}
\newtheorem{remark}[theorem]{Remark}

\numberwithin{equation}{theorem}
\newcommand{\Pic}{\text{Pic}} 
\newcommand{\Nef}{\text{Nef}} 
\newcommand{\Eff}{\text{Eff}}
\begin{document}
\title{Algebraic hyperbolicity for surfaces in smooth projective toric threefolds with Picard rank 2 and 3}
\author{Sharon Robins}
\address{Department of Mathematics, Simon Fraser University, 8888 University Drive, Burnaby BC V5A1S6, Canada  }
\email{srobins@sfu.ca}

\begin{abstract}
Algebraic hyperbolicity serves as a bridge between differential geometry and algebraic geometry. Generally, it is difficult to show that a given projective variety is algebraically hyperbolic. However, it was established recently that a very general surface of degree at least five in projective space is algebraically hyperbolic. 
We are interested in generalizing the study of surfaces in projective space to surfaces in smooth projective toric threefolds with Picard rank 2 or 3. Following Kleinschmidt and Batyrev, we explore the combinatorial description of smooth projective toric threefolds with Picard rank 2 and 3. We then use Haase and Ilten’s method of finding algebraically hyperbolic surfaces in toric threefolds. As a result, we determine many algebraically hyperbolic surfaces in each of these varieties. 
\end{abstract}

\maketitle

\section{Introduction}

\subsection{Background}
Hyperbolicity has long been a topic of interest in the study of differential geometry. A smooth complex projective variety $X$ is said to be Brody hyperbolic if there is no nonconstant holomorphic map $f : \mathbb{C} \to X$ \cite{MR3735863}. The hyperbolicity of curves can be completely determined by their geometric genus. A curve is hyperbolic if and only if the geometric genus is at least 2. However, the determination of hyperbolicity becomes complicated as we go to higher dimensions. 

Demailly \cite{MR1492539} introduced the notion of algebraic hyperbolicity to improve the study of hyperbolicity. According to Demailly, a smooth complex projective variety $X$ is algebraically hyperbolic if there exists an ample divisor $H$ on $X$ and some $\epsilon >0$ such that any curve $C\subset X$ of geometric genus $g(C)$ satisfies
\begin{align*}
    2g(C)-2\geq \epsilon(C\cdot H).
\end{align*}
The notion of algebraic hyperbolicity was extended to singular projective varieties by Javanpeykar and Kamenova in \cite{JK20}. Furthermore, there are studies of algebraic hyperbolicity in pseudo-setting, see e.g \cite{JX20}.

Compared to hyperbolicity, it is easier to determine the algebraic hyperbolicity of varieties. Moreover, hyperbolicity and algebraic hyperbolicity are equivalent by Demailly’s conjecture that is a consequence of Green-Griffiths-Lang conjecture \cite{Lan86}.

It was established recently that a very general surface of degree at least five in $\mathbb{P}^3$ is algebraically hyperbolic \cite{MR1258918}, \cite{MR3949983}. In \cite[Theorem 1.4]{1903.02681} Chiantini and Lopez provide algebraically hyperbolic surfaces in $\mathbb{P}^4$. Besides that, Haase and Ilten \cite{1903.02681} initiated the study of algebraically hyperbolic surfaces in toric threefolds. They  ensure the existence of hyperbolic surfaces of higher degrees \cite[Theorem 1.2]{1903.02681}. Further to that, Coskun and Riedl \cite{coskun2019algebraic} expanded their previous work on $\mathbb{P}^3$ to any threefold $X$ admitting an action with an open orbit by some algebraic group $G$. Through a deeper analysis of some special cases they complete the open cases posted by Haase and Ilten and provide a  classification of algebraically hyperbolic surfaces in $\mathbb{P}^2\times \mathbb{P}^1, \mathbb{P}^1\times \mathbb{P}^1\times \mathbb{P}^1, \mathbb{P}(\mathcal{O}_{\mathbb{P}^1}\oplus \mathcal{O}_{\mathbb{P}^1}(r))\times \mathbb{P}^1 $ and the blow-up of $\mathbb{P}^3$ at a point.

\subsection{Approach and results}
In this paper, we study the algebraic hyperbolicity of surfaces in  smooth projective toric threefolds with Picard rank 2 and 3 by following the methods in \cite{1903.02681}. A description of all smooth projective threefolds with Picard rank 2 and 3 is given in Section 3. For each of these varieties first we will analyze the geometric genus of curves in the boundary using Lemma \ref{boundarycurves}. For explicit results see Lemma \ref{lemma7.1}. Also we will provide a configuration of divisors with connected sections for each of these varieties (Lemma \ref{lemma7.2}). We then combine the genus of curves in the boundary, the collection of divisors with connected sections, and Theorem \ref{Theorem5.14}.
Consequently, we almost completely classify the algebraically hyperbolic surfaces in all smooth projective toric varieties with Picard rank 2 and 3 (Theorem \ref{first}).

\subsection{Organization} In Section 2.1, we fix notation and recall basic facts about toric varieties. We discuss our main tools to show algebraic hyperbolicity in Section 2.2. We recall the classification of all smooth projective threefolds with Picard rank 2 and 3 in Section 3. Finally, we state our main results in Section 4 and in Section 5 we prove the results.

\textbf{Acknowledgements}: 
I am thankful to Nathan Ilten for his constant encouragement and motivating this paper. I am also grateful to Nathan Ilten and anonymous referee for diligent reading of earlier versions of this paper and useful comments.  

\section{Preliminaries}

\subsection{Toric varieties}

We start by reviewing toric varieties. For general facts about toric varieties, see \cite{MR2810322} or \cite{MR1234037}. Let $N\cong \mathbb{Z}^n$ be a lattice of rank $n$, and $M=\text{Hom}(N,\mathbb{Z})$ be the dual lattice. The pairing between $m\in M$ and $u \in N$ is denoted by $\langle m,u \rangle \in \mathbb{Z}$. Given a fan $\Sigma$ in $N_{\mathbb{R}}$, we can associate a toric variety $X_{\Sigma}$. The properties of a fan $\Sigma$ gives a lot of information about the geometry of the toric variety $X_{\Sigma}$. For example, $X_\Sigma$ is smooth or complete if and only if $\Sigma$ is smooth or complete respectively \cite[Theorem 3.19]{MR2810322}. 

We always assume a fan $\Sigma$ is smooth and complete. We denote the rays of $\Sigma$ by $\Sigma(1)$. Given a ray $\rho \in \Sigma(1)$, we denote $u_{\rho}$ the primitive generator of $\rho$. We denote the divisor corresponding to $\rho$ by $D_{\rho}$. Any torus invariant divisor $D$ may be written uniquely as a sum $D=\Sigma_{\rho \in \Sigma(1)}a_{\rho}D_{\rho}$. To $D$, we associate the polytope
\[P(D)=\{m\in M_{\mathbb{R}}: \langle m,u_{\rho} \rangle \geq -a_{\rho} \text{~for all~} \rho\in \Sigma(1)\}.\]

 A support function on $\Sigma$ is a function $\varphi:\mathbb{R}^n \rightarrow \mathbb{R} $ that is linear on each cone of $\Sigma$ and $\mathbb{Z}$ valued on $\mathbb{Z}^n$. To any divisor $D=\Sigma_{\rho \in \Sigma(1)}a_{\rho}D_{\rho}$, we define a function $\varphi_{D}$ such that
\begin{equation*}
    \varphi_{D}(u_{\rho})=-a_{\rho} \text{ for all } \rho\in \Sigma(1).
\end{equation*}
It can be extended uniquely to a support function of $\Sigma$. Furthermore, the support function $\varphi_D$ is convex if and only if $D$ is nef. Also note that being nef is the same as basepoint free for toric varieties. If $E,E'$ are two nef divisors then 
\[P(E)+P(E')=P(E+E').\]
We say a pair of nef divisors $(E,E')$ has integer decomposition property (IDP) if 
\[(P(E)\cap \mathbb{Z}^n)+(P(E')\cap \mathbb{Z}^n)=P(E+E')\cap \mathbb{Z}^n.\]

In this paper, we are only interested in smooth complete toric threefolds. It is known that $\mathbb{P}^3$ is the only smooth complete toric threefold with Picard rank 1. We will give the description of all smooth complete toric threefolds with Picard rank 2 and 3 in Section 3. For smooth complete toric threefolds of at most Picard rank 3 we have the following result on IDP.

\begin{theorem} [{\cite[Corollary 4.2]{MR2551605}},{\cite[Theorem 1.4]{robins2021integer}}]\label{IDPtheorem}
Let $X_{\Sigma}$ be a smooth complete toric threefold of at most Picard rank 3 and let $E,E'$ be two nef divisor on $X_\Sigma$. Then $(E,E')$ has the IDP.
\end{theorem}

\subsection{Algebraic hyperbolicity}

In this section, we briefly discuss the tools developed by Haase and Ilten \cite{1903.02681} to determine the algebraically hyperbolic surfaces in toric threefolds. Haase and Ilten define the property that a configuration of divisors has connected sections \cite[Definition 3.3]{1903.02681}. We don't give the definition here. Instead, we will discuss a criterion which guarantees a pair of divisors has connected sections (Proposition  \ref{connectedsection}).  

We will first translate the information of the fan into a toric ideal. For the rest of the paper, we always assume $X_{\Sigma}$ be a smooth complete toric threefold. Then we have the short exact sequence 
\begin{equation}\label{exact2}
     0{\longrightarrow} M \overset{i}\longrightarrow \mathbb{Z}^{\Sigma[1]} \overset{\pi}\longrightarrow \text{Pic}(X_{\Sigma}) \rightarrow 0,
\end{equation}
where the $\rho$ coordinate of the first map is given by $m\mapsto \langle m,u_{\rho} \rangle$. 
After choosing a basis $\{e_1^{*},e_2^{*},e_3^{*}\}$ for $M$ and a basis for $\text{Pic}(X_{\Sigma})$, we get the short exact sequence

\begin{equation}\label{exact1}
    0{\longrightarrow} \mathbb{Z}^3 \overset{i}\longrightarrow \mathbb{Z}^r \overset{\pi}\longrightarrow \mathbb{Z}^{k} \rightarrow 0.
\end{equation}
Note that $k=r-3$.  The cokernel of this exact sequence can be identified by the Picard group. We can represent the map $i$ by a matrix $A$ and $\pi$ by a matrix $B=(b_{ij})$. Consider the  $\mathbb{C}$-algebra homomorphism given by
\begin{align*}
   \alpha: \mathbb{C}[x_1,\cdots,x_r]&\rightarrow \mathbb{C}[y_1,y_1^{-1},
   \ldots,y_{k},y_k^{-1}]\\
    x_i &\mapsto y_1^{b_{1i}}\dotsm y_{k}^{b_{ki}}.
\end{align*}
We define the toric ideal $I_B=$ ker$(\alpha)$.

\begin{lemma}[{\cite[Proposition 1.1.9]{MR2810322}}]
The toric ideal $I_B$ associated with the matrix $B$ is the ideal in $\mathbb{C}[x_1,\ldots,x_r]$ generated by binomials $x^{v^{+}}-x^{v^{-}}$ for $v^{+},v^{-}\in \mathbb{Z}^{r}_{\geq0}$ with $Bv^{+}=Bv^{-}$. 
\end{lemma}

We identify a vector $v\in$ ker($B)\cap \mathbb{Z}^r$ with the binomial $x^{v^{+}}-x^{v^{-}}$ where 
$v_i^{+}=$max($v_i,0$) and $v_i^{-}=-$min($v_i,0$). 

\begin{definition}
We say that a subset $\mathcal{G}\subset \text{ker}(B)\cap \mathbb{Z}^r$ is a Markov basis if the corresponding binomials generate $I_B$.
\end{definition}

\begin{proposition}[{\cite[Proposition 4.5]{1903.02681}}]\label{connectedsection}
Let $(E,E')$ be an IDP pair of divisors on $X_{\Sigma}$. Set
\[\mathcal{G}:=\left(i\left(P(E')\right)\cap \mathbb{Z}^r\right)-\left(i\left(P(E')\right)\cap \mathbb{Z}^r\right).\]
If $\mathcal{G}$ is a Markov basis for $I_B$, then the configuration $(E+E',E)$ has connected sections.
\end{proposition}

\begin{theorem}[{\cite[Theorem 3.6]{1903.02681}}]\label{Theorem5.14}
Let $(D,E)$ be non-trivial basepoint free torus invariant divisors on a smooth complete toric threefold $X_{\Sigma}$. Assume that this configuration has connected sections and that $D$ is big. Let $S\in |D|$ be a very general surface and $C\subset S$ is any curve that is not contained in the toric boundary of $X_{\Sigma}$. Then the geometric genus $g$ of $C$ satisfies
\[2g-2\geq C\cdot(E+K_{X_{\Sigma}}).\]
\end{theorem}

Even though the above theorem gives a bound for most of the curves, we need to determine the genus of finitely many curves that lie in the boundary. Polytopes associated to the divisors again will play a crucial role here. 
\begin{lemma}[{\cite[Lemma 4.1]{1903.02681}}]\label{boundarycurves}
Let $S\in |D|$ be a general surface, where $D$ is a big and basepoint free divisor. If $C\subset S$ is an curve contained in the toric boundary of $X$, then $C=S\cap D_{\rho}$ for some $\rho \in \Sigma(1)$ corresponding to a face $F<P(D)$ on which the ray $\rho$ takes its minimum. Then, the geometric genus of $C$ equals the number of interior lattice points of $F$ viewed in a two dimensional ambient space.  
\end{lemma}

\begin{remark}
If we do not assume that $D$ is big, then the geometric genus of $C$ is at most the number of interior lattice points of $F$.  Since $D$ is non-trivial but not big, dim$(P(D))$ is either 1 or 2. In this case, there is always a ray $\rho$ such that it takes its minimum on a 1-$d$ face. Since it has no interior lattice points we can conclude that there is a genus 0 curve in the boundary.
\end{remark}

\begin{corollary}\label{corollary111}
Let $X_{\Sigma}$ be a smooth projective toric threefold. Then a very general surface $S \in |D|$ is algebraically
hyperbolic if it has the following two properties:
\begin{enumerate}
    \item All curves in the toric boundary have genus at least two.
    \item There exists an ample divisor $H$ and an $\epsilon_0>0$ such that any curve not in the toric boundary satisfies
    \[2g(C)-2 \geq \epsilon_0 (C\cdot H).\]
\end{enumerate}
\end{corollary}

\begin{proof}

Let $S$ be any very general surface in $|D|$. Let $C_1,\ldots, C_k$ be all finitely many curves in the boundary. 
Take $\epsilon$ to be minimum among $\epsilon_0$ and $\dfrac{1}{C_i\cdot H}$ for $i=1,\ldots,k$. Then, we have 
\[2g(C)-2 \geq \epsilon (C\cdot H),\] for every curve in $S$.

\end{proof}

\begin{remark}
If we are only interested in studying pseudo-algebraic hyperbolicity, we can forget about 
curves in the boundary, and Theorem \ref{Theorem5.14} is sufficient. 
\end{remark}

We will also use the following result.

\begin{theorem}[Generalized Noether-Lefschetz theorem]\label{Theorem5.17}
Let $X_{\Sigma}$ be a smooth threefold, and $D$ be a divisor such that $D+K_{X_{\Sigma}}$ is basepoint free. Then for a very general surface $S\in|D|$, the restriction map \textup{Pic}$(X_{\Sigma})\rightarrow$ \textup{Pic}$(S)$ is an isomorphism. 
\end{theorem}
\begin{proof}
This is a particular case of Theorem 1 from \cite{MR2567426}.
\end{proof}

\section{Classification of toric varieties with Picard rank 2 or 3}

In this section, we recall the fan description of all smooth complete toric threefolds with Picard rank 2 or 3 using primitive collections. The notion of primitive collections is introduced by Batyrev \cite{batyrev1991}, which make classifications and computations easier for toric varieties. We also know that every smooth complete toric variety of Picard rank at most 3 is necessarily projective \cite{MR1098923}.

\begin{definition}\label{primitive}
Let $\Sigma$ be a fan.
A subset $\mathscr{P}=\{\rho_1,\rho_2,\ldots,\rho_k\} \subset \Sigma(1)$ is called a primitive collection if $\mathscr{P}$ is not contained in a single cone of $\Sigma$, but every proper subset is. Let $\mathscr{P}$ be a primitive collection and $\sigma\in\Sigma$ be the cone of the smallest dimension containing $u_{\rho_1}+\cdots+u_{\rho_k}$. Then there exist a unique expression
\begin{equation}\label{primitiverel}
u_{\rho_1}+\cdots+u_{\rho_k}=\sum_{\rho\in \sigma(1)}c_{\rho}u_{\rho}, \text{~~~}c_{\rho}\in\mathbb{Z}_{>0}. 
\end{equation}
The equation (\ref{primitiverel}) is the primitive relation of $\mathscr{P}$. 
\end{definition}

\begin{definition}
A fan is called a splitting fan if there is no intersection between any two primitive collections.
\end{definition}

It is easy to check whether a torus invariant Cartier divisor $D$ is nef or ample using primitive collections.

\begin{theorem}[{\cite[Theorem 6.4.9]{MR2810322}}]\label{theoremconevex}
Let $X_\Sigma$ be a smooth projective variety and $D$ be a torus invariant divisor on $X_{\Sigma}$. Then
$D$ is nef if and only if it satisfies
    \[\varphi_{D}(u_{\rho_{1}}+\cdots+u_{\rho_{k}})\geq \varphi_{D}(u_{\rho_{1}})+\cdots+\varphi_{D}(u_{\rho_{k}})\]
    for all primitive collections $\mathscr{P}=\{\rho_1,\ldots,\rho_k\}$ of $\Sigma$. Similarly, $\varphi_D$ is ample 
    if and only if it satisfies
    \[\varphi_{D}(u_{\rho_{1}}+\cdots+u_{\rho_{k}})> \varphi_{D}(u_{\rho_{1}})+\cdots+\varphi_{D}(u_{\rho_{k}})\]
    for all primitive collections $\mathscr{P}=\{\rho_1,\ldots,\rho_k\}$ of $\Sigma$.
\end{theorem}

Kleinschmidt \cite{Kleinschmidt1988} classified all smooth complete toric varieties with Picard rank 2. It turns out that all such varieties are projectivization of decomposable bundles over a projective space of a smaller dimension. We are only interested in three dimensional cases and it has following description.

\begin{theorem}[{\cite[Theorem 7.3.7, Example 7.3.5]{MR2810322}}]\label{fanrank2}
Let $X_{\Sigma}$ be a smooth complete toric threefold with Picard rank 2. Then we have the following two cases.
\begin{itemize}
    \item $X_{\Sigma} \cong \mathbb{P}\left(\mathcal{O}_{\mathbb{P}^2}\oplus \mathcal{O}_{\mathbb{P}^2}(l)\right)$ with $l\geq0$. The fan $\Sigma$ has primitive ray generators $u_1=e_1,u_2=-e_1, u_3=e_2,u_4=e_3$ and $u_5=le_1-e_2-e_3$. The only primitive collections are $\{\rho_1,\rho_2\}$ and $\{\rho_3,\rho_4,\rho_5\}$. 
    \item $X_{\Sigma} \cong \mathbb{P}\left( \mathcal{O}_{\mathbb{P}^1}\oplus \mathcal{O}_{\mathbb{P}^1}(l_1)\oplus \mathcal{O}_{\mathbb{P}^1}(l_2)\right)$ with $l_2\geq l_1 \geq 0$. The fan $\Sigma$ has primitive ray generators $u_1=e_1,u_2=e_2, u_3=-e_1-e_2,u_4=e_3$ and $u_5=l_1e_1+l_2e_2-e_3$. The only primitive collections are $\{\rho_1,\rho_2,\rho_3\}$ and $\{\rho_4,\rho_5\}$.
\end{itemize}
\end{theorem}

\begin{remark}\label{remark11}
We can also conclude from \cite[Example 7.3.5]{MR2810322} that smooth complete two dimensional toric varieties with Picard rank 2 are Hirzebruch surfaces  $\mathcal{H}_r=\mathbb{P}(\mathcal{O}_{\mathbb{P}^1}\oplus \mathcal{O}_{\mathbb{P}^1}(r))$. The fan $\Sigma$ has primitive ray generators $u_1=e_1,u_2=-e_1+re_2,u_3=e_2,u_4=-e_2$. The only primitive collections are $\{\rho_1,\rho_2\}$ and $\{\rho_3,\rho_4\}$. It is easy to see that $D_1$ and $D_4$ generate the Picard group using (\ref{exact2}).
\end{remark}

Batyrev completely classified smooth complete toric varieties with Picard rank 3 in terms of primitive collections. He showed that the number of primitive collections of its generators is 3 or 5 \cite[Theorem 5.7]{batyrev1991}. The case of 3 primitive collections is as in the following theorem.

\begin{theorem}\label{fan3}
Let $X_{\Sigma}$ be a smooth complete toric threefold of Picard rank 3 with a splitting fan. Then $X_{\Sigma} \cong \mathbb{P}(\mathcal{O}_{\mathcal{H}_r}\oplus \mathcal{O}_{\mathcal{H}_r}(aD_1+bD_4))$ with $a\geq0$. The fan $\Sigma$ has primitive ray generators $u_1=e_1,u_2=-e_1+re_2+ae_3,u_3=e_2,u_4=-e_2+be_3,u_5=e_3,u_6=-e_3$.  The only primitive collections are $\{\rho_1,\rho_2\},\{\rho_3,\rho_4\}$ and $\{\rho_5,\rho_6\}$.
\end{theorem}

\begin{proof}
In this case, the associated toric variety is isomorphic to the projectivization of a decomposable bundle of rank 2 over a smooth complete toric surface with Picard rank 2 \cite{batyrev1991}.  By Remark \hyperref[remark11]{\ref*{remark11}}, it is clear that smooth complete surfaces are given by Hirzebruch surfaces $\mathcal{H}_{r}$. Hence, we can take $\mathcal{E}=\mathcal{O}_{\mathcal{H}_r}(a_1D_1+b_1D_4)\oplus \mathcal{O}_{\mathcal{H}_r}(a_2D_1+b_2D_4)$. Without loss of generality, take $a_2\geq a_1$ and $\mathcal{L}= \mathcal{O}_{\mathcal{H}_r}(-a_1D_1-b_1D_4))$. Then $\mathbb{P}(\mathcal{E})\cong \mathbb{P}({\mathcal{E}}\otimes \mathcal{L})$ by Lemma 7.9 \cite{MR0463157} and first part of the result follows.  The fan description follows from Proposition \cite[Proposition 7.3.3]{MR2810322}. 

\end{proof}

For smooth complete toric varieties with Picard rank 3 and 5 primitive collections we have the following result.

\begin{theorem}[{\cite[Theorem 6.6]{batyrev1991}}]\label{Batyrev}
Let $X_{\Sigma}$ be a smooth projective toric threefold with Picard rank 3 which does not have a splitting fan. Then its ray generators can be partitioned into 5 non-empty sets $Y_0,Y_1,\cdots,Y_4$ in such a way that the primitive collections are $Y_i\cup Y_{i+1}$, where $i\in \mathbb{Z}/5\mathbb{Z}$. 
\end{theorem}

Let $\Sigma$ be a fan as in Theorem \ref{Batyrev}. Following \cite{batyrev1991}, we can list all the possibilities. Let
\begin{align*}
    Y_0&=\{v_1,\ldots,v_{p_0}\},    &  Y_1&=\{y_1,\ldots,y_{p_1}\}, &     Y_2=\{z_1,\ldots,z_{p_2}\},\\
    Y_3&=\{t_1,\ldots,t_{p_3}\},    & Y_4&=\{u_1,\ldots,u_{p_4}\}.
\end{align*}
Then arranging the primitive ray generators of $\Sigma$ in rows of the matrix $A$ we have the following 5 cases:

\begin{equation*}
    \begin{split}
        \begin{pNiceMatrix}[first-col]
        v_1&1 & 0 &0 \\ 
        v_2&0 &1 &0 \\
        u_1&-1&-1&b_1\\
        y_1&-1&-1&b_1+1\\
        t_1&0&0&1\\
        z_1&0&0&-1
        \end{pNiceMatrix},
    \end{split}
    \quad
    \begin{split}
        \begin{pNiceMatrix}[first-col]
        v_1&1 & 0 &0 \\ 
        u_1&-1 &0 &b_1 \\
        y_1&-1&-1&b_1+1\\
        y_2&0&1&0\\
        t_1&0&0&1\\
        z_1&0&0&-1
        \end{pNiceMatrix},
    \end{split}
    \quad
    \begin{split}
        \begin{pNiceMatrix}[first-col]
        v_1&1 & 0 &0 \\ 
        u_1&-1 &b_1 &c_2 \\
        y_1&-1&b_1+1&c_2\\
        t_1&0&1&0\\
        z_1&0&-1&-1\\
        z_2&0&0&1
        \end{pNiceMatrix},
    \end{split}
\end{equation*}    

\begin{equation*}
    \begin{split}
    \begin{pNiceMatrix}[first-col]
        v_1&1 & 0 &0 \\ 
        u_1&-1 &b_1 &b_2 \\
        y_1&-1&b_1+1&b_2+1\\
        t_1&0&1&0\\
        t_2&0&0&1\\
        z_1&0&-1&-1
        \end{pNiceMatrix},
    \end{split}    
    \begin{split}
       \begin{pNiceMatrix}[first-col]
        v_1&1 & 0 &0 \\ 
        u_1&-1 &-1 &b_1 \\
        u_2&0&1&0\\
        y_1&-1&0&b_1+1\\
        t_1&0&0&1\\
        z_1&0&0&-1
        \end{pNiceMatrix}. 
    \end{split}
\end{equation*}

\section{Main results on algebraic hyperbolicity}

In this section, we present our results on the algebraic hyperbolic surfaces in the smooth projective threefolds with Picard rank 2 and 3.
We will first study (\ref{exact1}) for each case. For convenience we will label different cases as follows:
\begin{itemize}
    \item Case 2.0.1.  $X_{\Sigma}\cong \mathbb{P}(\mathcal{O}_{\mathbb{P}^2}\oplus \mathcal{O}_{\mathbb{P}^2}(l))$.
    \item Case 2.0.2. $X_{\Sigma} \cong \mathbb{P}(\mathcal{O}_{\mathbb{P}^1}\oplus \mathcal{O}_{\mathbb{P}^1}(l_1)\oplus \mathcal{O}_{\mathbb{P}^1}(l_2))$. 
    \item Case 3.0.1. The fan $\Sigma$ in Theorem \ref{fan3} with $b\geq0$.
    \item Case 3.0.2. The fan $\Sigma$ in Theorem \ref{fan3} with $b<0$.
    \item Case 3.1.1. The fan $\Sigma$ in Theorem \ref{Batyrev} with $n=3,p_0=2$.
    \item Case 3.1.2. The fan $\Sigma$ in Theorem \ref{Batyrev} with $n=3,p_1=2$.
    \item Case 3.1.3. The fan $\Sigma$ in Theorem \ref{Batyrev} with $n=3,p_2=2$.
    \item Case 3.1.4. The fan $\Sigma$ in Theorem \ref{Batyrev} with $n=3,p_3=2$.
    \item Case 3.1.5. The fan $\Sigma$ in Theorem \ref{Batyrev} with $n=3,p_4=2$.
\end{itemize}

Here the first number in a case represents the Picard rank, second number 0 means it is associated with a splitting fan and 1 means not and the last number enumerates different cases if the first two numbers match. For Proposition \ref{prop6.1} and Proposition \ref{proposition6.2}, we will only give details for Case 2.0.1. The argument for the other cases is similar to Case 2.0.1.

\begin{proposition}\label{prop6.1}
The matrices $A,B$ in (\ref{exact1}) for each case and Markov basis for $B$ as columns of a matrix are listed in Table \ref{Markov}.
\end{proposition}
\begin{proof}
We will only give details for Case 2.0.1. Also note that the description of the matrices $A,B$ are follows the description of $\Sigma$.

Let $I=\langle x_3-x_4,x_3-x_5,x_1x_5^{l}-x_2  \rangle$ and $I_B$ be the toric ideal associated with the matrix $B$. It is enough to show that $I=I_B$. Clearly, $I\subset I_B$. Note that
\begin{align*}
    \dfrac{\mathbb{C}[x_1,x_2,x_3,x_4,x_5]}{\langle x_3-x_4,x_3-x_5,x_1x_5^{l}-x_2 \rangle}&\cong \dfrac{\mathbb{C}[x_1,x_2,x_3,x_5]}{\langle x_3-x_5,x_1x_5^{l}-x_2 \rangle}\\
    &\cong \dfrac{\mathbb{C}[x_1,x_2,x_5]}{\langle x_1x_5^{l}-x_2 \rangle}\\
    &\cong \mathbb{C}[x_1,x_5].
\end{align*}
Hence, $I$ is a prime ideal. Moreover,    dim($\mathbb{C}[x_1,...,x_5]/I)=2=$ rank($B$). Hence, $I=I_B$ by 
\cite[Lemma 4.2]{MR1363949}.

\end{proof}

\begin{table}
\tiny
\caption{Markov basis for Proposition \ref{prop6.1}}
\label{Markov}
\begin{tabular}{lll} \toprule
 \textbf{$A$} & \textbf{$B$} & \textbf{Markov basis} \\ \midrule
Case 2.0.1 &  & \\\\
$\begin{pNiceMatrix}[first-col]
\rho_1&1 & 0 &0 \\ 
\rho_2&-1 &0 &0 \\
\rho_3&0&1&0\\
\rho_4&0&0&1\\
\rho_5&l&-1&-1
\end{pNiceMatrix}$
& 
$\begin{pNiceMatrix}[first-row, first-col]
&D_1 &D_2 &D_3&D_4&D_5 \\ 
[D_2]&1 & 1&0&0&0 \\ 
[D_3]&-l & 0&1&1&1 \\
\end{pNiceMatrix}$
& 
$\begin{pNiceMatrix}
1 & 0 &0 \\ 
-1 &0 &0 \\
0&1&0\\
0&0&1\\
l&-1&-1
\end{pNiceMatrix}$
\\\\ \midrule

Case 2.0.2  & &\\\\
$\begin{pNiceMatrix}[first-col]
\rho_1&1 & 0 &0 \\ 
\rho_2&0 &1 &0 \\
\rho_3&-1&-1&0\\
\rho_4&0&0&1\\
\rho_5&l_1&l_2&-1
\end{pNiceMatrix}$
& 
$\begin{pNiceMatrix}[first-row, first-col]
&D_1 &D_2 &D_3&D_4&D_5 \\ 
[D_3]&1 & 1&1&0&0 \\ 
[D_4]&-l_1 &-l_2&0&1&1 \\
\end{pNiceMatrix}$
& 
$\begin{pNiceMatrix}
1 & 0 &0 \\ 
-1 &1 &0 \\
0&-1&0\\
0&0&1\\
l_1-l_2&l_2&-1
\end{pNiceMatrix}$
\\\\ \midrule

Case 3.0.1 and 3.0.2  & &\\\\
$\begin{pNiceMatrix}[first-col]
\rho_1&1 & 0 &0 \\ 
\rho_2&-1 &r &a \\
\rho_3&0&1&0\\
\rho_4&0&-1&b\\
\rho_5&0&0&1\\
\rho_6&0&0&-1
\end{pNiceMatrix}$
& 
$\begin{pNiceMatrix}[first-row, first-col]
&D_{1} &D_{2} &D_{3}&D_{4}&D_{5}&D_{6} \\ 
[D_{1}]&1 & 1&-r&0&-a&0 \\ 
[D_{4}]&0 & 0&1&1&-b&0\\
[D_{6}]&0&0&0&0&1&1
\end{pNiceMatrix}$
& 
$\begin{pNiceMatrix}
1 & 0 &0 \\ 
-1 &r &a+br \\
0&1&b\\
0&-1&0\\
0&0&1\\
0&0&-1
\end{pNiceMatrix}$
\\\\ \midrule

Case 3.1.1  & &\\\\
$\begin{pNiceMatrix}[first-col]
v_1&1 & 0 &0 \\ 
v_2&0 &1 &0 \\
u_1&-1&-1&b_1\\
y_1&-1&-1&b_1+1\\
t_1&0&0&1\\
z_1&0&0&-1
\end{pNiceMatrix}$
& 
$\begin{pNiceMatrix}[first-row, first-col]
&D_{v_1} &D_{v_2} &D_{u_1}&D_{y_1}&D_{t_1}&D_{z_1} \\ 
[D_{v_1}]&1 & 1&0&1&-b-1&0 \\ 
[D_{u_1}]&0 & 0&1&-1&1&0\\
[D_{z_1}]&0&0&0&0&1&1
\end{pNiceMatrix}$
& 
$\begin{pNiceMatrix}
1 & 0 &0 \\ 
0 &1 &0 \\
-1&-1&b_1\\
-1&-1&b_1+1\\
0&0&1\\
0&0&-1
\end{pNiceMatrix}$
\\\\ \midrule

Case 3.1.2  & &\\\\
$\begin{pNiceMatrix}[first-col]
v_1&1 & 0 &0 \\ 
u_1&-1 &0 &b_1 \\
y_1&-1&-1&b_1+1\\
y_2&0&1&0\\
t_1&0&0&1\\
z_1&0&0&-1
\end{pNiceMatrix}$
& 
$\begin{pNiceMatrix}[first-row,first-col]
&D_{v_1} & D_{u_1}&D_{y_1}&D_{y_2}&D_{t_1}&D_{z_1} \\ 
[D_{v_1}]&1 & 0&1&1&-b-1&0 \\ 
[D_{u_1}]&0 & 1&-1&-1&1&0\\
[D_{z_1}]&0&0&0&0&1&1
\end{pNiceMatrix}$
& 
$\begin{pNiceMatrix}
 1&0 &0 \\ 
-1&0 &b_1 \\
-1&-1&b_1+1\\
-0&1&0\\
0&0&1\\
0&0&-1
\end{pNiceMatrix}
$
\\\\ \midrule

Case 3.1.3  & &\\\\
$\begin{pNiceMatrix}[first-col]
v_1&1 & 0 &0 \\ 
u_1&-1 &b_1 &c_2 \\
y_1&-1&b_1+1&c_2\\
t_1&0&1&0\\
z_1&0&-1&-1\\
z_2&0&0&1
\end{pNiceMatrix}$
& 
$\begin{pNiceMatrix}[first-row,first-col]
&D_{v_1} & D_{u_1}&D_{y_1}&D_{t_1}&D_{z_1}&D_{z_2} \\ 
[D_{v_1}]&1 & 0&1&-b_1-1&0&-c_2 \\ 
[D_{u_1}]&0 & 1&-1&1&0&0\\
[D_{z_1}]&0&0&0&1&1&1
\end{pNiceMatrix}$
& 
$\begin{pNiceMatrix}
1 & 0 &0 \\ 
-1 &b_1 &b_1-c_2 \\
-1&b_1+1&b_1+1-c_2\\
0&1&1\\
0&-1&0\\
0&0&-1
\end{pNiceMatrix}
$
\\\\ \midrule

Case 3.1.4  & &\\\\
$\begin{pNiceMatrix}[first-col]
v_1&1 & 0 &0 \\ 
u_1&-1 &b_1 &b_2 \\
y_1&-1&b_1+1&b_2+1\\
t_1&0&1&0\\
t_2&0&0&1\\
z_1&0&-1&-1
\end{pNiceMatrix}$
& 
$\begin{pNiceMatrix}[first-row,first-col]
&D_{v_1} & D_{u_1}&D_{y_1}&D_{t_1}&D_{t_2}&D_{z_1} \\ 
[D_{v_1}]&1 & 0&1&-b_1-1&-b_2-1&0 \\ 
[D_{u_1}]&0 & 1&-1&1&1&0\\
[D_{z_1}]&0&0&0&1&1&1
\end{pNiceMatrix}$
& 
$\begin{pNiceMatrix}
1 & 0 &0 \\ 
-1 &b_1 &b_1-b_2 \\
-1&b_1+1&b_1-b_2\\
0&1&1\\
0&0&-1\\
0&-1&0
\end{pNiceMatrix}
$
\\\\ \midrule

Case 3.1.5  & &\\\\
$\begin{pNiceMatrix}[first-col]
v_1&1 & 0 &0 \\ 
u_1&-1 &-1 &b_1 \\
u_2&0&1&0\\
y_1&-1&0&b_1+1\\
t_1&0&0&1\\
z_1&0&0&-1
\end{pNiceMatrix}$
& 
$\begin{pNiceMatrix}[first-row,first-col]
&D_{v_1} & D_{u_1}&D_{u_2}&D_{y_1}&D_{t_1}&D_{z_1} \\ 
[D_{v_1}]&1 & 0&0&1&-b_1-1&0 \\ 
[D_{u_1}]&0 & 1&1&-1&1&0\\
[D_{z_1}]&0&0&0&0&1&1
\end{pNiceMatrix}$
& 
$\begin{pNiceMatrix}
1 & 0 &0 \\ 
-1 &-1 &b_1 \\
0&1&0\\
-1&0&b_1+1\\
0&0&1\\
0&0&-1
\end{pNiceMatrix}
$
\\\\

\bottomrule
\end{tabular}

\end{table}

\begin{figure}[ht]
\centering
\resizebox{0.7\textwidth}{!}{%
\begin{subfigure}{.5\textwidth}
  \centering
  \begin{tikzpicture}[scale=0.5]
\draw[thick,-] (-.5,0)--(5,0);
\draw[thick,-] (0,-.5)--(0,5);
\fill[gray] (0,0) rectangle (4,4);
\filldraw[black] (3,0) circle (2pt) node[anchor=north] {$[D_2]$};
\filldraw[black] (0,3) circle (2pt) node[anchor=east] {$[D_3]$};
\end{tikzpicture}
  \caption{Nef cone}
  
\end{subfigure}%
\begin{subfigure}{.5\textwidth}
  \centering
  \begin{tikzpicture}[scale=0.5]
    \draw[thick,-] (-.5,0)--(5,0);
    \draw[thick,-] (0,-.5)--(0,5);
    \draw[thick,-] (0,0)--(5,-5);
     \fill[gray] (0,0) rectangle (4,4);
    \draw[fill=gray] (0,0)--(4,0)--(4,-4);
    \filldraw[black] (3,0) circle (2pt) node[anchor=north] {$[D_2]$};
    \filldraw[black] (0,3) circle (2pt) node[anchor=east] {$[D_3]$};
    \filldraw[black] (3,-3) circle (2pt) node[anchor=east] {$[D_1]$};
   
    \end{tikzpicture}
  \caption{Effective cone}
  
\end{subfigure}
}%
\caption{Nef and effective cone of $X_\Sigma\cong \mathbb{P}(\mathcal{O}_{\mathbb{P}^2}\oplus \mathcal{O}_{\mathbb{P}^2}(l))$ }
\label{fig:test1}
\end{figure}
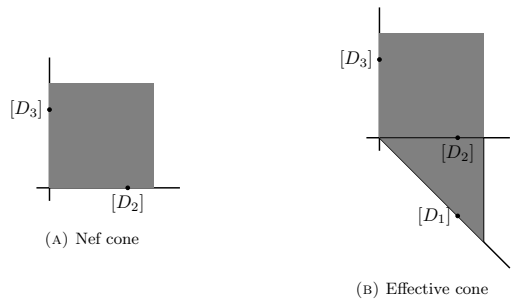

\begin{proposition}\label{proposition6.2}
Let $X_\Sigma$ be a smooth complete toric threefold with Picard rank 2 or 3. In Table \ref{Pic}, we collect the basis of Picard group $\emph{Pic}(X_\Sigma)$, primitive ray generators of nef cone $\emph{Nef}(X_{\Sigma})$ and effective cone $\emph{Eff}(X_{\Sigma})$, and in the last column we give a representative for the canonical divisor $K_{X_{\Sigma}}$.
\end{proposition}
\begin{proof}
We will only give details for Case 2.0.1. By (\ref{exact2}),
the Picard group is generated by the classes of $D_i$ subject to the following relations:
\begin{equation}\label{eq1}
\begin{split}
 D_1-D_2+lD_3 &\sim 0\\
    D_3-D_5 &\sim 0\\
    D_4-D_5 &\sim 0. 
\end{split}    
\end{equation}
Thus, $\Pic(X)= \mathbb{Z}[D_2] \bigoplus \mathbb{Z}[D_3]$.
Let $D=aD_2+bD_3$. 
If $D$ is nef then we have (Theorem \ref{theoremconevex}
), 

\begin{equation*}
\begin{split}
    \varphi_D(0)&\geq   \varphi_D(e_1)+\varphi_D(-e_1)\\
    \varphi_D(le_1)&\geq   \varphi_D(e_2)+\varphi_D(e_3)+\varphi_D(le_1-e_2-e_3)
\end{split}    
\end{equation*}
and it follows that $a,b\geq0$. Thus, the nef cone is generated by the classes of $D_2$ and $D_3$. Using (\ref{eq1}), one can easily see that the effective cone is generated by the classes of $D_1$ and $D_3$. See Figure \ref{fig:test1}.
The canonical divisor of $X_{\Sigma}$ is given by (\cite[Theorem 8.2.3]{MR2810322})
\begin{align*}
    K_{X}&=-D_1-D_2-D_3-D_4-D_5\\
    &\sim -2D_2+(l-3)D_3.
\end{align*}
\end{proof}

\begin{table}
\tiny
\caption{Basis of Picard group $\text{Pic}(X_\Sigma)$, primitive ray generators of nef cone $\text{Nef}(X_{\Sigma})$, effective cone $\text{Eff}(X_{\Sigma})$ and a representative for the canonical divisor $K_{X_{\Sigma}}$ for Proposition \ref{proposition6.2}
}
\label{Pic}
\begin{tabular}{lllll} \toprule
\\
\textbf{$X_{\Sigma}$} & \textbf{$\Pic(X_{\Sigma})$} & \textbf{$\Nef(X_{\Sigma})$} &\textbf{$\Eff(X_{\Sigma})$} & \textbf{$K_{X_{\Sigma}}$}\\\\ \midrule
\\
Case 2.0.1  & $D_2,D_3$  & $D_2,D_3$ & $D_1,D_3$ & $-2D_2+(l-3)D_3$  \\
\midrule
\\
Case 2.0.2 & $D_3,D_4$  & $D_3,D_4$ & $D_2,D_4$ & $-3D_3+(l_1+l_2-2)D_4$ \\
\midrule
\\
Case 3.0.1 & $D_1,D_4,D_6$  & $D_1,D_4,D_6$ & $D_1,D_3,D_5$ & $(-2+a+r)D_1+(-2+b)D_4-2D_6$ \\
\midrule
\\
\multirow{2}{*}{Case 3.0.2} & \multirow{2}{*}{$D_1,D_4,D_6$}  & \multirow{2}{*}{$D_1,D_4,D_6-bD_4$} & $D_1,D_3,D_6$ if $b<0$  & \multirow{2}{*}{$(-2+a+r)D_1+(-2+b)D_4-2D_6$} \\\\
 &   &  & and $a+br\leq0$ &  \\\\

&  & & $D_1,D_3,D_5,D_6$ if $b<0$&  \\\\
&  & &  and $a+br>0$ &  \\\\
\midrule
\\
Case 3.1.1 & $D_{v_1},D_{z_1},D_{u_1}$  & $D_{v_1},D_{z_1},D_{u_1}+D_{z_1}$ & $D_{u_1},D_{y_1},D_{t_1}$ & $(b_1-2)D_{v_1}-D_{u_1}-2D_{z_1}$ \\
\midrule
\\
Case 3.1.2 & $D_{v_1},D_{z_1},D_{u_1}$  & $D_{v_1},D_{z_1},D_{u_1}+D_{z_1}$ & $D_{u_1},D_{y_1},D_{t_1}$ & $(b-2)D_{v_1}-2D_{z_1}$ \\
\midrule
\\
\multirow{2}{*}{Case 3.1.3} & \multirow{2}{*}{$D_{v_1},D_{z_1},D_{u_1}$}  & \multirow{2}{*}{$D_{v_1},D_{z_1},D_{u_1}+D_{z_1}$} & $D_{u_1},D_{y_1},D_{t_1}$ if $b_1\geq c_2$ & \multirow{2}{*}{$(b_1+c_2-1)D_{v_1}--D_{u_1}-3D_{z_1}$} \\\\
 &   &  & $D_{u_1},D_{y_1},D_{z_2}$ if $b_1< c_2$ &  \\
 \midrule
 \\
\multirow{2}{*}{Case 3.1.4} & \multirow{2}{*}{$D_{v_1},D_{z_1},D_{u_1}$}  & \multirow{2}{*}{$D_{v_1},D_{z_1},D_{u_1}+D_{z_1}$} & $D_{u_1},D_{y_1},D_{t_1}$ if $b_1\geq b_2$ & \multirow{2}{*}{$(b_1+b_2)D_{v_1}-2D_{u_1}-3D_{z_1}$ }\\\\
 &   &  & $D_{u_1},D_{y_1},D_{t_2}$ else & \\
\midrule 
\\
Case 3.1.5 & $D_{v_1},D_{z_1},D_{u_1}$  & $D_{v_1},D_{z_1},D_{u_1}+D_{z_1}$ & $D_{u_1},D_{y_1},D_{t_1}$ & $(b_1-1)D_{v_1}-2D_{u_1}-2D_{z_1}$ \\\\
\bottomrule
\\
\end{tabular}

\end{table}

\begin{theorem}\label{first}
Let $X_\Sigma$ be a smooth complete toric threefold with Picard rank 2 or 3 and let $S$ be a very general surface of class $D$ in $X_{\Sigma}$. A classification of cases when $S$ is algebraically hyperbolic is found in Tables \ref{tab:long2},\ref{tab:long3},\ref{tab:long4}.
\end{theorem}
\begin{remark}
The algebraic hyperbolicity of very general surfaces in some of the smooth complete toric threefolds with Picard rank 2 or 3 are previously studied by Haase and Ilten \cite{1903.02681} and Coskun and Riedl \cite{coskun2019algebraic}.
When $l=0$ in Case 2.0.1 or $l_1=l_2=0$ in Case 2.0.2 then $X_{\Sigma}\cong \mathbb{P}^2 \times \mathbb{P}^1$. See the results in {\cite[Example 6.1]{1903.02681}}, {\cite[Section 3.2]{coskun2019algebraic}}.
When $l=1$ in Case 2.0.1 then $X_{\Sigma}$ is isomorphic to the blow-up of $\mathbb{P}^3$ at a point. See the results in {\cite[Example 6.4]{1903.02681}}, {\cite[Section 3.4]{coskun2019algebraic}}. 
When $r=a=b=0$ in Case 3.0.1 then $X_{\Sigma}\cong \mathbb{P}^1 \times \mathbb{P}^1 \times \mathbb{P}^1$. See the results in {\cite[Example 6.3]{1903.02681}}, {\cite[Section 3.1]{coskun2019algebraic}}.
When $a=b=0$ and $r\geq1$ in Case 3.0.1 then $X_{\Sigma}\cong {\mathcal{H}_r} \times \mathbb{P}^1$. See the results in  {\cite[Section 3.3]{coskun2019algebraic}}.
\end{remark}

We will prove Theorem \ref{first} in the next section.

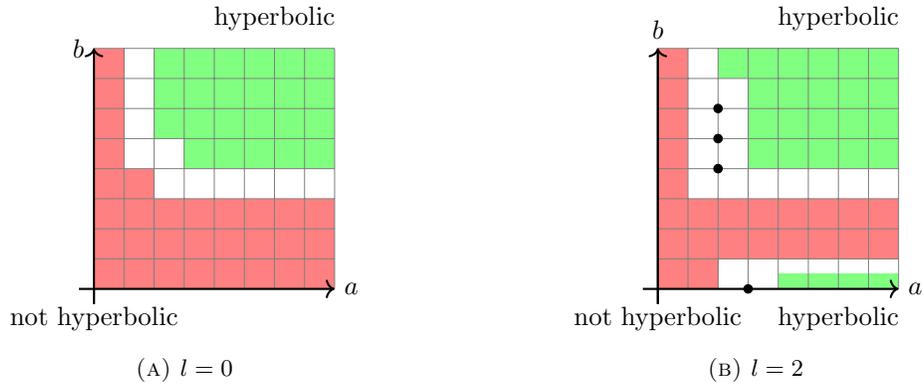
\begin{figure}[htp!]

\begin{subfigure}{.4\textwidth}
    \centering    
    \begin{tikzpicture}[scale=.4]
    \draw[thick,-] (0,-.5)--(0,8);
    \fill[red!50] (0,0) rectangle (1,8);
     \fill[red!50] (0,0) rectangle (2,4);
    \fill[red!50] (0,0) rectangle (8,3);
    \fill[green!50] (2,5) rectangle (8,8);
    \fill[green!50] (3,4) rectangle (8,8);
    \draw[step=1cm,gray,very thin] (0,0) grid (8,8);
    
    \draw[thick,->] (-.5,0)--(8,0) node(xaxis) [right] {$a$};
    \draw[thick,->] (0,-.5)--(0,8) node (yaxis) [left] {$b$};
    \node at (6,9){hyperbolic};
    \node at (0,-1){not hyperbolic};
    \end{tikzpicture}
    \caption{$l=0$}
\end{subfigure}\hfill
\begin{subfigure}{.4\textwidth}

    \begin{tikzpicture}[scale=.4]
    \draw[thick,-] (0,-.5)--(0,8);
    \fill[red!50] (0,0) rectangle (1,8);
    \fill[red!50] (0,1) rectangle (8,3);
    \fill[red!50] (1,0) rectangle (2,1);
    \fill[green!50] (3,4) rectangle (8,8);
    \fill[green!50] (4,0) rectangle (8,.5);
    \fill[green!50] (2,7) rectangle (8,8);
    \draw[step=1cm,gray,very thin] (0,0) grid (8,8);
    \filldraw[black] (2,4) circle (4pt);
    \filldraw[black] (3,0) circle (4pt);
    \filldraw[black] (2,5) circle (4pt);
    \filldraw[black] (2,6) circle (4pt);
    \draw[thick,->] (-.5,0)--(8,0) node(xaxis) [right] {$a$};
    \draw[thick,->] (0,-.5)--(0,8) node (yaxis) [above] {$b$};
     \node at (6,9){hyperbolic};
     \node at (0,-1){not hyperbolic};
     \node at (6,-1){hyperbolic};
    \end{tikzpicture}
    \caption{$l=2$}
\end{subfigure}

    \caption{Algebraic hyperbolicity for a very general surface of the type $aD_2+bD_3$ in $X_\Sigma\cong \mathbb{P}(\mathcal{O}_{\mathbb{P}^2}\oplus \mathcal{O}_{\mathbb{P}^2}(l))$, for $l=0$ and 2.}
\label{fig:my_label_1}

\end{figure}

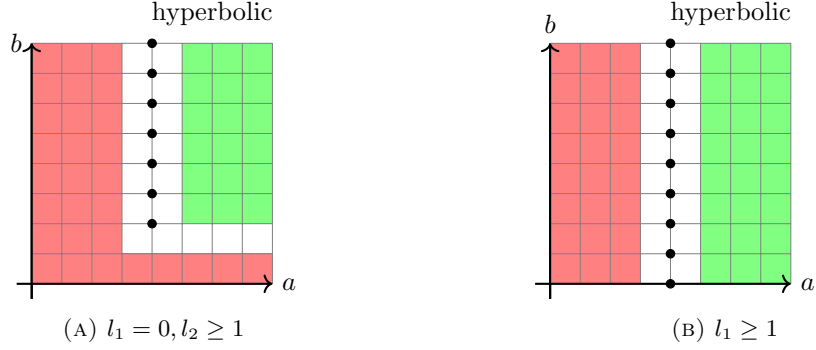
\begin{figure}[ht]

\begin{subfigure}{.4\textwidth}
    \centering    
    \begin{tikzpicture}[scale=.4]
    \draw[thick,-] (0,-.5)--(0,8);
    \fill[red!50] (0,0) rectangle (3,8);
    \fill[red!50] (3,0) rectangle (8,1);
    \fill[green!50] (5,2) rectangle (8,8);
    \draw[step=1cm,gray,very thin] (0,0) grid (8,8);
    \filldraw[black] (4,2) circle (4pt);
    \filldraw[black] (4,3) circle (4pt);
    \filldraw[black] (4,4) circle (4pt);
    \filldraw[black] (4,5) circle (4pt);
    \filldraw[black] (4,6) circle (4pt);
     \filldraw[black] (4,7) circle (4pt);
    \filldraw[black] (4,8) circle (4pt);
    \draw[thick,->] (-.5,0)--(8,0) node(xaxis) [right] {$a$};
    \draw[thick,->] (0,-.5)--(0,8) node (yaxis) [left] {$b$};
    \node at (6,9){hyperbolic};
    \end{tikzpicture}
    \caption{$l_1=0,l_2\geq1$}
    \label{fig:sub3}
\end{subfigure}\hfill
\begin{subfigure}{.4\textwidth}

    \begin{tikzpicture}[scale=.4]
    \draw[thick,-] (0,-.5)--(0,8);
    \fill[red!50] (0,0) rectangle (3,8);
    \fill[green!50] (5,0) rectangle (8,8);
    \draw[step=1cm,gray,very thin] (0,0) grid (8,8);
    \filldraw[black] (4,0) circle (4pt);
    \filldraw[black] (4,1) circle (4pt);
    \filldraw[black] (4,2) circle (4pt);
    \filldraw[black] (4,3) circle (4pt);
    \filldraw[black] (4,4) circle (4pt);
     \filldraw[black] (4,5) circle (4pt);
    \filldraw[black] (4,6) circle (4pt);
    \filldraw[black] (4,7) circle (4pt);
    \filldraw[black] (4,8) circle (4pt);
    \draw[thick,->] (-.5,0)--(8,0) node(xaxis) [right] {$a$};
    \draw[thick,->] (0,-.5)--(0,8) node (yaxis) [above] {$b$};
     \node at (6,9){hyperbolic};
    \end{tikzpicture}
    \caption{$l_1\geq1$}
    \label{fig:sub4}
\end{subfigure}

    \caption{Algebraic hyperbolicity for a very general surface of the type $aD_3+bD_4$ in $X_{\Sigma} \cong \mathbb{P}(\mathcal{O}_{\mathbb{P}^1}\oplus \mathcal{O}_{\mathbb{P}^1}(l_1)\oplus \mathcal{O}_{\mathbb{P}^1}(l_2))$}
\label{fig:my_label_2}
\end{figure}


{
\newpage
\tiny
\begin{longtable}{llllll}
\caption{Algebraic hyperbolicity of a very general surface $S$ in Theorem \ref{first}, when $X_{\Sigma}$ is a smooth complete toric threefold with  Picard rank 2 }
\label{tab:long2} \\
\toprule
\multicolumn{1}{c}{$X_{\Sigma}$}&
\multicolumn{1}{c}{$D$} & \multicolumn{1}{c}{\text{Extra conditions}} & \multicolumn{1}{c}{\text{Hyperbolic if}} &
\multicolumn{1}{c}{\text{ Not Hyperbolic if}}&
\multicolumn{1}{c}{\text{ Open case if}}\\\\
&& & \multicolumn{1}{c}{$(a,b)$} &
\multicolumn{1}{c}{$(a,b)$}&
\multicolumn{1}{c}{$(a,b)$}
\\\\ 
\hline 
\\

\multirow{23}{*}{Case 2.0.1} & \multirow{23}{*}{$aD_2+bD_3$} & \multirow{3}{*}{$l=0$\footnotemark[1]}  & \tabitem $(\geq3,\geq4)$ & \tabitem $(\leq1,\geq0)$ & \text{No case}   \\\\
& & & \tabitem $(=2,\geq5)$&  \tabitem $(\geq0,\leq3)$ & \\\\
& & & &  \tabitem $=(2,4)$ & \\\\

\cmidrule(r){3-6}

 & & \multirow{5}{*}{$l=1$\footnotemark[1]}  & \tabitem $(\geq3,\geq4)$ & \tabitem $(\leq1,\geq0)$ & \text{No case}   \\\\
& & &\tabitem $(=2,\geq5)$  & \tabitem $(\geq0,\in \{1,2,3\})$ & \\\\
& & & \tabitem $(\geq5,=0)$ &  \tabitem $(\leq4,=0)$ & \\\\
\cmidrule(r){3-6}

 & & \multirow{5}{*}{$l=2$}  & \tabitem $(\geq3,\geq4)$ & \tabitem $(\leq1,\geq0)$ & \tabitem $(=2,\in\{4,5,6\})$   \\\\
& & &\tabitem $(=2,\geq7)$  & \tabitem $(\geq0,\in \{1,2,3\})$ & \tabitem $=(3,0)$ \\\\
& & & \tabitem $(\geq4,=0)$ &  \tabitem $=(2,0)$ & \\\\
\cmidrule(r){3-6}

 & & \multirow{5}{*}{$l=3$}  & \tabitem $(\geq3,\geq4)$ & \tabitem $(\leq1,\geq0)$ & \tabitem $(=2,\in \{4,5,6\})$   \\\\
& & &\tabitem $(=2,\geq7)$  & \tabitem $(\geq0,\in \{1,2,3\})$ & \tabitem $(\in\{2,3\},=0)$ \\\\
& & & \tabitem $(\geq4,=0)$ &   & \\\\
\cmidrule(r){3-6}

 & & \multirow{5}{*}{$l\geq4$}  & \tabitem $(\geq3,\geq4)$ & \tabitem $(\leq1,\geq0)$ & \tabitem $(=2,\in\{4,5,6\})$   \\\\
& & &\tabitem $(=2,\geq7)$  & \tabitem $(\geq0,\in \{1,2,3\})$ & \tabitem $=(2,0)$ \\\\
& & & \tabitem $(\geq3,=0)$ &   & \\\\

\midrule

\multirow{5}{*}{Case 2.0.2} & \multirow{5}{*}{$aD_3+bD_4$} & \multirow{3}{*}{$l_1=l_2=0$\footnotemark[1]}  & \tabitem $(\geq4,\geq3)$ & \tabitem $(\leq3,\geq0)$ & \text{No case}   \\\\
& & & \tabitem $(\geq5,=2)$ &  \tabitem $(\geq0,\leq1)$ & \\\\

& & &  &  \tabitem $=(4,2)$ & \\\\
\cmidrule(r){3-6}

& & {$l_1=0,l_2\geq 1$}  & \tabitem $(\geq5,\geq2)$ & \tabitem $(\leq3,\geq0)$& \tabitem $(=4,\geq2)$   \\\\ 
& &   &  & \tabitem $(\geq0,\leq1)$ &    \\\\ 
\cmidrule(r){3-6}

& & {$l_1\geq1$}  & \tabitem$(\geq5,\geq0)$  & \tabitem $(\leq3,\geq0)$ & \tabitem $(=4,\geq 0)$   \\\\

\bottomrule
\end{longtable}
}
\footnotetext[1]{These cases are previously studied.}



{\tiny
\begin{longtable}{lllll}
\caption{Algebraic hyperbolicity of a very general surface $S$ in Theorem \ref{first}, when $X_{\Sigma}$ is a smooth complete toric threefold with  Picard rank 3 and $\Sigma$ is a splitting fan 
}
 \label{tab:long3} \\
\hline 
\\
\multicolumn{1}{c}{\textbf{$X_{\Sigma}$ and $D$}} & \multicolumn{1}{c}{\textbf{Extra conditions}} & \multicolumn{1}{c}{\textbf{Hyperbolic if}} &
\multicolumn{1}{c}{\textbf{ Not Hyperbolic if}}&
\multicolumn{1}{c}{\textbf{ Open case if}}
\\\\ 

 & \multicolumn{1}{c}{$(r,a,b)$} & \multicolumn{1}{c}{$(d,e,f)$} &
\multicolumn{1}{c}{$(d,e,f)$}&
\multicolumn{1}{c}{$(d,e,f)$}
\\\\
\hline

\endfirsthead

\multicolumn{5}{c}%
{{\bfseries \tablename\ \thetable{} -- continued from previous page}} \\\\ \hline 
\\
\multicolumn{1}{c}{\textbf{$X_{\Sigma}$ and $D$}} & \multicolumn{1}{c}{\textbf{Extra conditions}} & \multicolumn{1}{c}{\textbf{Hyperbolic if}} &
\multicolumn{1}{c}{\textbf{ Not Hyperbolic if}}&
\multicolumn{1}{c}{\textbf{ Open case if}}
\\\\
& \multicolumn{1}{c}{$(r,a,b)$} & \multicolumn{1}{c}{$(d,e,f)$} &
\multicolumn{1}{c}{$(d,e,f)$}&
\multicolumn{1}{c}{$(d,e,f)$}
\\\\ \hline 
\\
\endhead

\hline \multicolumn{5}{|r|}{{Continued on next page}} \\ \hline
\endfoot

\hline \hline
\endlastfoot

\\
Case 3.0.1 &  \multirow{2}{*}{$=(0,0,0)$\footnotemark[2]} & \tabitem $( \geq3,\geq3,\geq3)$\footnotemark[3]  & \tabitem $(\leq1,\geq0,\geq0)$\footnotemark[3] &   \\\\
 $dD_1+eD_4+fD_6$& &\tabitem $(2,\geq4,\geq4)$\footnotemark[3]  & \tabitem $(=2,\leq3,\geq0)$ & No case \\\\
\cmidrule(r){2-5}

 & \multirow{2}{*}{$(\geq1,=0,=0)$\footnotemark[2]}& \tabitem $(\geq 2,\geq3,\geq3)$ & \tabitem $(\leq1,\geq0,\geq0)$\footnotemark[3] & \\\\
 & & \tabitem $(\geq 3,\geq4,=2)$ & \tabitem $(\geq0,=2,\in\{2,3\})$ & No case \\\\
 & &\tabitem $(\geq 2+\delta_{r,1},=2,\geq4)$  & \tabitem $(\geq0,=3,=2,)$ & \\\\
 & &  & \tabitem $(=2,\geq0,=2)$ & \\\\
 & &  & \tabitem $(=2,=2,\geq1)$ if $r=1$ & \\\\
\cmidrule(r){2-5}


 & \multirow{2}{*}{General}& If $e,f\geq2$ and not & \tabitem $(\leq0,\leq1,\geq0)$ & \\\\
 & & in the next column & \tabitem $(\geq0,\geq0,\leq1)$ & \footnotemark[4] \\\\

 & & \tabitem $(\geq 4-a-r,\geq4-b,\geq3)$ &\tabitem $(\geq0,=2,=2)$ if $b=0$  & \\\\

 & & \tabitem $(\geq 4-a-r ,\geq3-b,\geq4)$ &\tabitem $(=1,\geq0,\geq0)$ if $a=0$  & \\\\

 & & \tabitem $(\geq 3-a-r,\geq4-b,\geq4)$ & \tabitem $(=2,\geq0,=2)$ if $a=0$ & \\\\

 & &  & \tabitem $(\leq1,\geq0,=2)$ if $a=1$ & \\\\
 
 & &  & \tabitem $(=0,\geq0,=2)$ if $a=2$ & \\\\

 & &  & \tabitem $(\leq1,\geq0,\geq0)$ if $r=0$ & \\\\

 & &  & \tabitem $(=2,=2,\geq0)$ if $r=0$ & \\\\

 & &  & \tabitem $(\leq1,=2,\geq0)$ if $r=1$ & \\\\
 & &  & \tabitem $(=0,=2,\geq0)$ if $r=2$ & \\\\\\

\cmidrule(r){2-5}

& \multirow{2}{*}{$(\geq3,\geq3,\geq1)$}& \tabitem $(\geq0,\geq2,\geq3)$ & \tabitem $(\geq0,\leq1,\geq0)$ & \tabitem $(\geq0,\geq2,=2)$ \\\\
 & &  & \tabitem $(\geq0,\geq0,\leq1)$ &  \\\\
 
\midrule

\\
Case 3.0.2 & \multirow{2}{*}{$(0,0,\leq-1)$}& \tabitem $(\geq4,\geq2,\geq4)$ & \tabitem $(\geq0,\geq0,\leq1)$ & \tabitem $(=2,\geq3,\geq3)$\\\\
\multirow{2}{*}{$dD_1+(e-bf)D_4+fD_6$} & &  &\tabitem $(\geq0,\leq1,\geq0)$ &\tabitem $(=3,\geq2,\geq2)$ \\\\
 & &  &\tabitem $(\leq1,\geq2,\geq2)$  & \tabitem $(\geq4,\geq2,\in\{2,3\})$ \\\\
 & &  &\tabitem $(=2,=2,\geq2)$  &  \\\\
 & &  &\tabitem $(=2,\geq2,=2)$  &  \\\\
\cmidrule(r){2-5}

 & \multirow{2}{*}{$(0,\geq1,\leq-1)$}& \tabitem $(\geq4,\geq2,\geq4)$ & \tabitem $(\geq0,\geq0,\leq1)$ & \tabitem $(\in\{2,3\},\geq2,\geq2)$\\\\
 & &  &\tabitem $(\geq0,\leq1,\geq0)$ &\tabitem $(\geq4,\geq2,\in\{2,3\})$ \\\\
 & &  &\tabitem $(\leq1,\geq2,\geq2)$  &  \\\\
\cmidrule(r){2-5}

 & \multirow{2}{*}{$(\geq1,=0,\leq-1)$}& \tabitem $(\geq4,\geq2,\geq4)$ & \tabitem $(\geq0,\geq0,\leq1)$ & \tabitem $(=2,\geq2,\geq3)$\\\\
 & &  &\tabitem $(\geq0,\leq1,\geq0)$ & \tabitem$(=3,\geq2,\geq2)$ \\\\
 & &  &\tabitem $(\leq1,\geq2,\geq2)$  &\tabitem $(\geq4,\geq2,=3)$  \\\\
 & &  &\tabitem $(=2,\geq2,=2)$  &  \\\\

 & \multirow{2}{*}{$(\geq1,=1,\leq-1)$}& \tabitem $(\geq4,\geq2,\geq4)$ & \tabitem $(\geq0,\geq0,\leq1)$ & \tabitem $(\leq3,\geq2,\geq3)$\\\\
 & &  &\tabitem $(\geq0,\leq1,\geq0)$ &\tabitem $(\geq4,\geq2,\in\{2,3\})$ \\\\
 & &  &\tabitem $(\leq1,\geq2,=2)$  &  \\\\
\cmidrule(r){2-5} 

 & \multirow{2}{*}{$(\geq1,=2,\leq-1)$}& \tabitem $(\geq4,\geq2,\geq4)$ & \tabitem $(\geq0,\geq0,\leq1)$ &\tabitem $(=0,\geq2,\geq3)$ \\\\
 & &  &\tabitem $(\geq0,\leq1,\geq0)$ &\tabitem $(\in\{1,2,3\},\geq2,\geq2)$ \\\\
 & &  &\tabitem $(=0,\geq2,=2)$  & \tabitem $(\geq4,\geq2,\in\{2,3\})$ \\\\
\cmidrule(r){2-5}

 & \multirow{2}{*}{$(\geq1,\geq3,\leq-1)$}& \tabitem $(\geq4,\geq2,\geq4)$ & \tabitem $(\geq0,\geq0,\leq1)$ &\tabitem $(\leq3,\geq2,\geq2)$ \\\\
 & &  &\tabitem $(\geq0,\leq1,\geq0)$ &\tabitem $(\geq4,\geq2,\in\{2,3\})$ \\\\

\bottomrule

\end{longtable}
}
\footnotetext[2]{These cases are previously studied.}
\footnotetext[3]{It also holds for any permutation of these numbers.}
\footnotetext[4]{The open cases are difficult to list in the given space.}


{\tiny
\begin{longtable}{lllll}
\caption{Algebraic hyperbolicity for a very general surface $S$ of class $D=dD_{v_1}+fD_{u_1}+(e+f)D_{z_1}$ in Theorem \ref{first}, when $X_{\Sigma}$ is a smooth complete toric threefold with  Picard rank 3 and $\Sigma$ is not a splitting fan 
}

 \label{tab:long4} \\
\hline 
\\
\multicolumn{1}{c}{$X$} & \multicolumn{1}{c}{\textbf{Extra conditions}} & \multicolumn{1}{c}{\textbf{Hyperbolic if}} &
\multicolumn{1}{c}{\textbf{ Not Hyperbolic if}}&
\multicolumn{1}{c}{\textbf{ Open case if}}
\\\\ 

 &  & \multicolumn{1}{c}{$(d,e,f)$} &
\multicolumn{1}{c}{$(d,e,f)$}&
\multicolumn{1}{c}{$(d,e,f)$}
\\\\
\hline

\endfirsthead

\multicolumn{5}{c}%
{{\bfseries \tablename\ \thetable{} -- continued from previous page}} \\\\ \hline 
\\
\multicolumn{1}{c}{\textbf{X}} & \multicolumn{1}{c}{\textbf{Extra conditions}} & \multicolumn{1}{c}{\textbf{Hyperbolic if}} &
\multicolumn{1}{c}{\textbf{ Not Hyperbolic if}}&
\multicolumn{1}{c}{\textbf{ Open case if}}
\\\\
&  & \multicolumn{1}{c}{$(d,e,f)$} &
\multicolumn{1}{c}{$(d,e,f)$}&
\multicolumn{1}{c}{$(d,e,f)$}
\\\\ \hline 

\endhead

\hline \multicolumn{5}{|r|}{{Continued on next page}} \\ \hline
\endfoot

\hline \hline
\endlastfoot

\\
\multirow{6}{*}{Case 3.1.1}  & \multirow{3}{*}{$b_1=0$}  & \tabitem $(\geq4,\geq3,\geq2)$  & \tabitem $(\in \{1,2,3\},\geq0,\geq0)$   & \tabitem$(\geq4,=2,=2)$   \\\\
 & &\tabitem $(\geq4,=2,\geq3)$  & \tabitem $(\geq0,\geq0,\leq1)$   &\tabitem $(\geq4,=0,\in\{2,3,4\})$\\\\
 & &\tabitem$(\geq4,=0,\geq5)$  & \tabitem $(\geq0,=1,\geq0)$  & \tabitem $(=0,\geq2,\geq4)$\\\\
 
& &  & \tabitem $(=0,=0,\leq3)$ &\tabitem $(0,=0,\geq4)$\\\\
& &  & \tabitem $(=0,\geq2,\leq3)$ &\\\\

\cmidrule(r){2-5}

  & \multirow{3}{*}{$b_1=1$}  & \tabitem $(\geq4,\geq3,\geq2)$  & \tabitem $(\in \{1,2,3\},\geq0,\geq0)$   & \tabitem$(\geq4,=2,=2)$   \\\\
 & &\tabitem $(\geq4,=2,\geq3)$  & \tabitem $(\geq0,\geq0,\leq1)$   &\tabitem $(\geq4,=0,\in\{2,3,4\})$\\\\
 & &\tabitem$(\geq4,=0,\geq5)$  & \tabitem $(\geq0,=1,\geq0)$  & \tabitem $(=0,\geq2,\geq3)$\\\\
 
& &  & \tabitem $(=0,=0,\leq2)$ &\tabitem $(0,=0,\geq3)$\\\\
& &  & \tabitem $(=0,\geq2,\leq2)$ &\\\\
\cmidrule(r){2-5}

  & \multirow{3}{*}{$b_1\geq2$}  & \tabitem $(\geq4,\geq3,\geq2)$  & \tabitem $(\in \{1,2,3\},\geq0,\geq0)$   & \tabitem$(\geq4,=2,=2)$   \\\\
 & &\tabitem $(\geq4,=2,\geq3)$  & \tabitem $(\geq0,\geq0,\leq1)$   &\tabitem $(\geq4,=0,\in\{2,3,4\})$\\\\
 & &\tabitem$(\geq4,=0,\geq5)$  & \tabitem $(\geq0,=1,\geq0)$  & \tabitem $(=0,\geq2,\geq2)$\\\\
 
& &  &  &\tabitem $(0,=0,\geq2)$\\\\

\cmidrule(r){1-5}

\\\\\\\\

\multirow{6}{*}{Case 3.1.2}  & \multirow{3}{*}{$b_1=0$}  & \tabitem $(\geq4,\geq4,\geq1)$  & \tabitem $(\in \{1,2,3\},\geq0,\geq0)$ & \tabitem $(\geq4,\geq4,=0)$   \\\\
 & &  & \tabitem$(\geq0,\in\{1,2,3\},\geq0)$  & \tabitem $(0,\geq4,\geq2)$\\\\
  & &  & \tabitem $(0,\geq4,\leq1)$  & \tabitem $(\geq4,=0,\geq2)$ \\\\
 & &  & \tabitem  $(\geq4,0,\leq1)$ &\tabitem $(0,0,\geq4)$  \\\\
 & &  & \tabitem  $(0,0,\leq3)$ &  \\\\
\cmidrule(r){2-5}

& \multirow{3}{*}{$b_1=1$}  & \tabitem $(\geq4,\geq4,\geq1)$  & \tabitem $(\in\{1,2,3\},\geq0,\geq0)$   & \tabitem $(\geq4,\geq4,=0)$ \\\\
 & &  & \tabitem$(\geq0,\in\{1,2,3\},\geq0)$    & \tabitem$(0,\geq4,\geq0)$\\\\
 & &  &\tabitem  $(\geq4,=0,<=1)$  &\tabitem$(\geq4,=0,\geq2)$\\\\
 & &  &\tabitem  $(=0,=0,<=2)$  &\tabitem$(=0,=0,\geq3)$\\\\
\cmidrule(r){2-5}

& \multirow{3}{*}{$b_1\geq2$}  & \tabitem $(\geq4,\geq4,\geq1)$  & \tabitem $(\in\{1,2,3\},\geq0,\geq0)$   & \tabitem $(\geq4,\geq4,=0)$ \\\\
 & &  & \tabitem$(\geq0,\in\{1,2,3\},\geq0)$    & \tabitem$(0,\geq4,\geq0)$\\\\
 & &  &\tabitem  $(\geq4,=0,<=1)$  &\tabitem$(\geq4,=0,\geq2)$\\\\
 & &  &\tabitem  $(=0,=0,<=1)$  &\tabitem$(=0,=0,\geq2)$\\\\
\cmidrule(r){1-5}

\multirow{2}{*}{Case 3.1.3}& $c_2=0$  & \tabitem $(\geq2,\geq4,\geq2)$  & \tabitem $(\geq0,\in\{1,2,3\},\geq0)$   & $(\geq2,=0,\geq4)$   \\\\
 & &  &\tabitem  $(\geq0,\geq0,\leq1)$&\\\\
 & &  &\tabitem  $(\leq1,\geq4,\geq2)$&\\\\
 & &  &\tabitem  $(\geq0,=0,\leq3)$&\\\\
 & &  &\tabitem  $(\leq1,=0,\geq0)$&\\\\
\cmidrule(r){2-5}

& $b_1=0,c_2=1$  & \tabitem $(\geq1,\geq4,\geq2)$   & \tabitem $(\geq0,\in\{1,2,3\},\geq0)$   & \tabitem$(=0,\geq4,\geq2)$   \\\\
 & &  &\tabitem  $(\geq0,\geq0,\leq1)$& \tabitem$(\geq0,=0,\geq4)$\\\\
& &  &\tabitem  $(\geq0,=0,\leq3)$&\\\\
\cmidrule(r){2-5}

& \tabitem $b_1=0,c_2\geq2$  & \tabitem $(\geq0,\geq4,\geq2)$   & \tabitem $(\geq0,\in\{1,2,3\},\geq0)$   &\tabitem$(\geq0,=0,\geq4)$   \\\\
 &\tabitem $b_1\geq1,c_2\geq1$ &  &\tabitem  $(\geq0,\geq0,\leq1)$& \\\\
& &  &\tabitem  $(\geq0,=0,\leq3)$&\\\\
\midrule

\multirow{2}{*}{Case 3.1.4}  & \multirow{2}{*}{$b_1=b_2=0$}  & \tabitem$(\geq1,\geq2,\geq4)$   & \tabitem $(\geq0,\geq0,\in\{1,2,3\})$ & \tabitem $(=0,\geq2,\geq4)$  \\\\
 & &  &\tabitem$(\geq0,\leq1,\geq0)$  &\tabitem $(\geq2,\geq4,=0)$\\\\
  & &  &\tabitem $(\geq0,\leq3,=0)$  &\\\\
   & &  &\tabitem $(\leq1,\geq0,=0)$  &\\\\
\cmidrule(r){2-5}

 & \tabitem $b_1\geq1,b_2=0$  & \tabitem$(\geq0,\geq2,\geq4)$   & \tabitem $(\geq0,\geq0,\in\{1,2,3\})$ & \tabitem  $(\geq2,\geq4,=0)$  \\\\
 &  \tabitem $b_1=0,b_2\geq1$& &\tabitem$(\geq0,\leq1,\geq0)$  &\\\\
 & &  &\tabitem $(\geq0,\leq3,=0)$  &\\\\
 & &  &\tabitem $(\leq1,\geq0,=0)$  &\\\\ 
\cmidrule(r){2-5}

&  $b_1>=1,b_2>=1$  & \tabitem$(\geq0,\geq2,\geq4)$   & \tabitem $(\geq0,\geq0,\in\{1,2,3\})$ & \tabitem  $(\geq0,\geq4,=0)$  \\\\
 & & &\tabitem$(\geq0,\leq1,\geq0)$  &\\\\\\
 & &  &\tabitem $(\geq0,\leq3,=0)$  &\\\\
  
\midrule

\multirow{2}{*}{Case 3.1.5}  & \multirow{2}{*}{-}  & \tabitem $(\geq2,\geq0,\geq5)$  & \tabitem $(\geq0,\geq0,\in\{1,2,3\})$  & \tabitem $(\geq2,=0,=4)$   \\\\
&   &\tabitem $(\geq2,\geq1,=4)$  & \tabitem $(\in\{0,1\},\geq0,\geq0)$  & \tabitem $(\geq4,\geq2,=0)$  \\\\
&   &  & \tabitem $(\in\{2,3\},\geq0,=0)$  &   \\\\
&   &  & \tabitem $(\geq0,\in\{0,1\},=0)$  &   \\\\

\bottomrule
\end{longtable}
}

\section{Proof of main results}
First we analyze the genus 0 or 1 curves in the boundary using Lemma \ref{boundarycurves}. For Lemma \ref{lemma7.1} and Lemma \ref{lemma7.2}, we will only give details for Case 2.0.1. The argument for the other cases is similar to Case 2.0.1.

\begin{lemma}\label{lemma7.1}
Let $X_\Sigma$ be a smooth complete toric threefold with Picard rank 2 or 3 and let $S$ be a very general surface in $X_{\Sigma}$. In Tables \ref{tab:long2},\ref{tab:long3},\ref{tab:long4}  the cases in the column not algebraically hyperbolic are the necessary conditions for the the curves in boundary has geometric genus at most 1.
\end{lemma}
\begin{proof}
We prove the case 2.0.1. Analysis of the remaining cases are similar. Since $a\geq1$, $D$ is big. 
We will analyze each of the facets of $P(D)$. See Figure \ref{fig:Facetscase1}. 
\begin{enumerate}
    \item \textbf{Facet 1.}
    The interior is given by the equations $x=0$, $y>-b$, $z>0$ and $lx-y-z>0$. It is easy to find the integer values $x,y,z$ that satisfy the above equations. Indeed there are $(b-1)(b-2)/2$ solutions if $b\geq1$. If $b=0$, the restriction is the trivial divisor, and there is no curve in the intersection.
    \item \textbf{Facet 2.}
    The interior is given by the equations $x=a$, $y>-b$, $z>0$ and $la-y-z>0$. Thus, we have $(la+b-2)(la+b-1)/2$ interior lattice points.
    \item \textbf{Facet 3.}
    The interior is given by the equations $y=-b$, $0<x<a$, $z>0$ and $lx+b-z>0$. Thus, we have $(a-1)(b-1)+la(a-1)/2$ interior lattice points.
    \item \textbf{Facet 4.}
    The interior is given by the equations $z=0$, $0<x<a$, $y>-b$ and $lx>y$. Thus, we have $(a-1)(b-1)+la(a-1)/2$ interior lattice points.
    \item \textbf{Facet 5.}
    The interior is given by the equations $lx-y-z=0$, $0<x<a$, $y>-b$ and $z>0$. Thus, we have $(a-1)(b-1)+la(a-1)/2$ interior lattice points.
\end{enumerate}
Thus by Lemma \ref{boundarycurves}, we have the results.
\end{proof}


\begin{figure}[ht]
    \begin{subfigure}[b]{0.32\textwidth}
        \centering
        \resizebox{\linewidth}{!}{
            \begin{tikzpicture}
            \draw[blue, very thick] (0,0) -- (3,0);
            \draw[blue, very thick] (0,0) -- (0,3);
            \draw[blue, very thick] (3,0) -- (0,3);
            \filldraw[red] (0,0) circle (2pt) node[anchor=north] {$(0,-b,0)$};
            \filldraw[red] (3,0) circle (2pt) node[anchor=north] {$(0,0,0)$};
            \filldraw[red] (0,3) circle (2pt) node[anchor=south] {$(0,-b,b)$}; 
        \end{tikzpicture}
        }
        \caption{Facet-1}
         \label{fig:Facet1}
    \end{subfigure}
    \begin{subfigure}[b]{0.32\textwidth}
    \centering
        \resizebox{\linewidth}{!}{\begin{tikzpicture}
            \draw[blue, very thick] (0,0) -- (3,0);
            \draw[blue, very thick] (0,0) -- (0,3);
            \draw[blue, very thick] (3,0) -- (0,3);
            \filldraw[red] (0,0) circle (2pt) node[anchor=north] {$(a,-b,0)$};
            \filldraw[red] (3,0) circle (2pt) node[anchor=north] {$(a,la,0)$};
            \filldraw[red] (0,3) circle (2pt) node[anchor=south] {$(a,-b,la+b)$}; 
        \end{tikzpicture}
            
        }
       \caption{Facet-2}
         \label{fig:Facet2}
    \end{subfigure}
    \begin{subfigure}[b]{0.32\textwidth}
        \centering
        \resizebox{\linewidth}{!}{
            \begin{tikzpicture}
            \draw[blue, very thick] (0,0) -- (2,0);
            \draw[blue, very thick] (0,0) -- (0,2);
            \draw[blue, very thick] (2,0) -- (2,4);
            \draw[blue, very thick] (0,2) -- (2,4);
            \filldraw[red] (0,0) circle (2pt) node[anchor=north] {$(0,-b,0)$};
            \filldraw[red] (2,0) circle (2pt) node[anchor=north] {$(a,-b,0)$};
            \filldraw[red] (0,2) circle (2pt) node[anchor=east] {$(0,-b,b)$};
            \filldraw[red] (2,4) circle (2pt) node[anchor=south] {$(a,-b,la+b)$};
        \end{tikzpicture}

        }
        \caption{Facet-3}
        \label{fig:Facet3}
    \end{subfigure}
    \\
    \begin{subfigure}[b]{0.45\textwidth}
        \centering
        \resizebox{\linewidth}{!}{
            \begin{tikzpicture}
            \draw[blue, very thick] (0,0) -- (2,0);
            \draw[blue, very thick] (0,0) -- (0,2);
            \draw[blue, very thick] (2,0) -- (4,2);
            \draw[blue, very thick] (0,2) -- (4,2);
            \filldraw[red] (0,0) circle (2pt) node[anchor=north] {$(0,-b,0)$};
            \filldraw[red] (2,0) circle (2pt) node[anchor=north] {$(0,0,0)$};
            \filldraw[red] (0,2) circle (2pt) node[anchor=east] {$(a,-b,0)$};
            \filldraw[red] (4,2) circle (2pt) node[anchor=south] {$(a,la,0)$};
        \end{tikzpicture}

        }
        \caption{Facet-4}
        \label{fig:Facet4}
    \end{subfigure}
    \begin{subfigure}[b]{0.45\textwidth}
        \centering
        \resizebox{\linewidth}{!}{
            \begin{tikzpicture}
            \draw[blue, very thick] (0,0) -- (2,0);
            \draw[blue, very thick] (0,0) -- (0,2);
            \draw[blue, very thick] (2,0) -- (2,4);
            \draw[blue, very thick] (0,2) -- (2,4);
            \filldraw[red] (0,0) circle (2pt) node[anchor=north] {$(0,-b,b)$};
            \filldraw[red] (2,0) circle (2pt) node[anchor=west] {$(a,-b,la+b)$};
            \filldraw[red] (0,2) circle (2pt) node[anchor=east] {$(0,0,0)$};
            \filldraw[red] (2,4) circle (2pt) node[anchor=south] {$(a,la,0)$};
        \end{tikzpicture}

        }
        \caption{Facet-5}
        \label{fig:Facet5}
    \end{subfigure}

\caption{Facets of Polytope $P(aD_2+bD_3)$ when $a,b\geq1$ in Lemma \ref{lemma7.1}}
\label{fig:Facetscase1}
\end{figure}
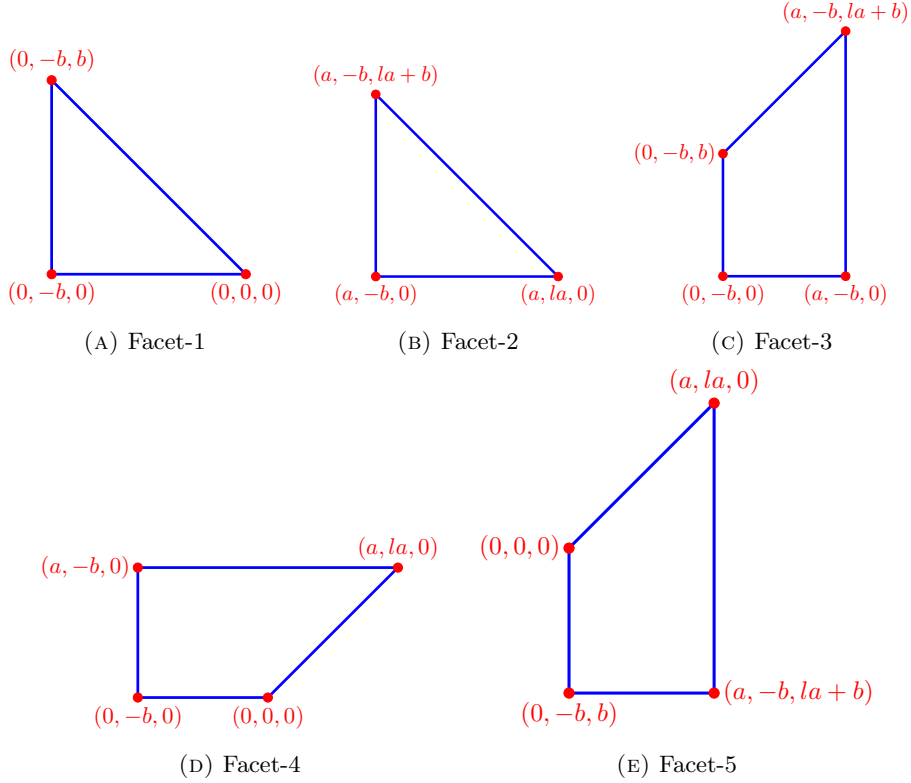

\begin{lemma}\label{lemma7.2}
Let $X_\Sigma$ be a smooth complete toric threefold with Picard rank 2 or 3. In Table \ref{tab:conn}, we list the divisors $E,E'$ in each case such that the configuration $(E+E',E)$ has connected sections.
\end{lemma}
\begin{proof}

Let 
\begin{equation*}
\begin{split}
A=
\begin{pmatrix}
1&0&0\\
-1&0&0\\
0&1&0\\
0&0&1\\
l&-1&-1
\end{pmatrix},
\end{split}
\quad
\begin{split}
    t_1=
    \begin{pmatrix}
    0\\
    -1\\
    0\\
    0\\
    0
    \end{pmatrix},
\end{split}
\quad
\begin{split}
    t_2=
    \begin{pmatrix}
    0\\
    -a\\
    -b\\
    0\\
    0
    \end{pmatrix},
\end{split}
\quad
\begin{split}
    t_3=
    \begin{pmatrix}
    0\\
    -a-1\\
    -b\\
    0\\
    0
    \end{pmatrix}.
\end{split}    
\end{equation*}
Then $P(E')=P(A,t_1)$, $P(D)=P(A,t_2)$ and $P(E)=P(A,t_3)$. Clearly the polytope $P(E')$ contains the points $(0,0,0),(1,0,0),(1,l-1,1),(1,l,0),(1,l-1,0)\in \mathbb{Z}^3$.
Let $G=i(P(E')-i(P(E'))$. Then $G$ contains the vectors 
\begin{align*}
    i(1,0,0)-i(0,0,0))&=(1,-1,0,0,l)-(0,0,0,0,0)=(1,-1,0,0,l)\\
     i(1,l,0)-i(1,l-1,1))&=(1,-1,l,0,0)-(1,-1,l-1,1,0)=(0,0,1,-1,0)\\
    i(1,l-1,1)-i(1,l-1,0)&=(1,-1,l-1,1,0)-(1,-1,l-1,0,1)=(0,0,0,1,-1).
\end{align*}
Thus by Proposition \ref{prop6.1}, $G$  is a Markov basis for $I_B$. Note that $P(D)\cap \mathbb{Z}^3=P(E)\cap\mathbb{Z}^3+P(E')\cap\mathbb{Z}^3$ by Theorem \ref{IDPtheorem}. Hence by Proposition \ref{connectedsection}, configuration $(D,E)$ has connected sections.
\end{proof}


\begin{table}
\tiny
\caption{ The divisors $E,E'$ such that the configuration $(E+E',E)$ has connected sections in Lemma \ref{lemma7.2}}
\label{tab:conn}
\begin{tabular}{llll}
\toprule
\multicolumn{1}{c}{$X$}&
\multicolumn{1}{c}{$E$} & \multicolumn{1}{c}{\text{Extra conditions}} & \multicolumn{1}{c}{$E'$} \\\\
\hline
\\

\text{Case 2.1.1} & $aD_2+bD_3$ & $l=0$  & $D_2+D_3$     \\\\ 
\cmidrule(r){3-4}
&   & $l\geq1$ &  $D_2$  \\\\ 
\midrule
\text{Case 2.1.2} & $aD_3+bD_4$ & $l_1=l_2=0$  & $D_3+D_4$     \\\\ 
\cmidrule(r){3-4}
&   & $l_2\geq1$ &  $D_4$  \\\\ 
\midrule

\multirow{3}{*}{Case 3.0.1} & \multirow{3}{*}{$dD_1+eD_4+fD_6$}& $(r,a,b)=(0,0,0)$  & $D_1+D_4+D_6$ \\\\ 
\cmidrule(r){3-4}
&   & $b=0,r+a\geq1$ &  $D_4+D_6$\\\\ 
\cmidrule(r){3-4}
&   & $b\geq1,r+a=0$ &  $D_1+D_6$\\\\ 
\cmidrule(r){3-4}
&   & $b\geq1,r+a\geq1$ &  $D_6$\\\\
\midrule

\multirow{3}{*}{Case 3.0.2} & \multirow{3}{*}{$dD_1+eD_4+f(D_6-bD_4)$}& $r+a=0$  & $D_1+D_6-bD_4$ \\\\ 
\cmidrule(r){3-4}
&   & $r+a\geq1$ &  $D_6-bD_4$\\\\
\midrule

\multirow{3}{*}{Case 3.1.1} &\multirow{3}{*}{$dD_{v_1}+fD_{u_1}+(e+f)D_{z_1}$}& $b_1=0$  & \tabitem $D_{u_1}+D_{z_1}$     \\\\ 
& &   & \tabitem $D_{v_1}+D_{z_1}$ \\\\
\cmidrule(r){3-4}
&   & $b_1>1$ &  $D_{z_1}$  \\\\ 
\midrule

\multirow{3}{*}{Case 3.1.2} & \multirow{3}{*}{$dD_{v_1}+fD_{u_1}+(e+f)D_{z_1}$}& \multirow{2}{*}{$b_1=0$}  & \tabitem $D_{u_1}+D_{z_1}$    \\\\ 
 & &   & \tabitem $D_{v_1}+D_{z_1}$ \\\\  
\cmidrule(r){3-4}
&   & $b_1>1$ &   $D_{z_1}$  \\\\ 
\midrule

\multirow{3}{*}{Case 3.1.3} & \multirow{3}{*}{$dD_{v_1}+fD_{u_1}+(e+f)D_{z_1}$}&\multirow{2}{*} {$b_1=c_2=0$}  & \tabitem $D_{u_1}+D_{z_1}$   \\\\ 
 & &   & \tabitem$D_{v_1}+D_{z_1}$ \\\\  
\cmidrule(r){3-4}
&   & \text{else} &  $D_{z_1}$  \\\\ 
\midrule

\multirow{3}{*}{Case 3.1.4} & \multirow{3}{*}{$dD_{v_1}+fD_{u_1}+(e+f)D_{z_1}$}& \multirow{2}{*}{$b_1=b_2=0$}  & \tabitem $D_{u_1}+D_{z_1}$   \\\\ 
 & &   & \tabitem $D_{v_1}+D_{z_1}$ \\\\  
\cmidrule(r){3-4}
&   & else &  $D_{z_1}$  \\\\ 
\midrule

\multirow{3}{*}{Case 3.1.5} & \multirow{3}{*}{$dD_{v_1}+fD_{u_1}+(e+f)D_{z_1}$}& \multirow{2}{*}{$b_1=0$}  & \tabitem $D_{u_1}+D_{z_1}$    \\\\ 
 & &   & \tabitem $D_{v_1}+D_{z_1}$ \\\\  
\cmidrule(r){3-4}
&   & $b_1>1$ &  $D_{z_1}$  \\\\

\bottomrule
\end{tabular}

\end{table}


\begin{corollary}
Let $X_\Sigma$ be a smooth complete toric threefold with Picard rank 2 or 3. Suppose $E$ and $E'$ are nef divisors on $X_\Sigma$ with $E'$ is big. Then $(E+E',E)$ has connected sections.
\end{corollary}
\begin{proof}[Proof of Theorem \ref{first}]
We will consider each case separately.\\

\textbf{Case 2.0.1.~~} 
The case $l=0$ is discussed in \cite[Example 6.1]{1903.02681} and  \cite[Theorem 1.1]{MR3949983}. Hence, we can assume $l\geq1$.
Let $D=aD_2+bD_3$ with $a\geq1$ and $b\geq0$ and $E=(a-1)D_2+bD_3$. Then $D$ is big, and the configuration $(D,E)$ has connected sections by Lemma \ref{lemma7.2}. Applying Theorem \ref{Theorem5.14}, for any curve $C$ not contained in the toric boundary on a very general surface $S$ in $|D|$ we have, 
\begin{align}\label{case21eq}
    2g-2 \geq C\cdot ((a-3)D_2+(b+l-3)D_3).
\end{align}
By Theorem \ref{Theorem5.17}, if $a\geq2$ and $b\geq 3-l$, then the natural restriction map Pic$(X_{\Sigma}) \rightarrow$ Pic$(S)$ is an isomorphism.
Thus, any curve $C$ is rationally equivalent to the complete intersection of $S$ with a divisor in the class $cD_1+dD_3$. If $C$ is not contained in the boundary, an intersection number calculation yields
\begin{equation*}
     2g-2\geq
    cb(b+l-3)+d(la(a-3)+a(b+l-3)+(a-3)b).
\end{equation*}
The degree of such a curve $C$ with respect to ample class $H=D_2+D_3$ is given by
\begin{equation*}
    \text{deg}(C)= cb+d(a+b+al).
\end{equation*}
Let $\epsilon_0=\dfrac{1}{a+b+al}$. Then we have
\begin{equation*}
    2g-2\geq\epsilon_0\cdot \text{deg}(C).
\end{equation*}
By combining Lemma \ref{lemma7.1} and Corollary \ref{corollary111}, we have the results.\\\\


\textbf{Case 2.0.2.~~} 
The case $l_1=l_2=0$ is discussed in \cite[Example 6.1]{1903.02681}. Hence, we can assume $l_2\geq1$
Let $D=aD_3+bD_4$ with $a\geq1$ and $b\geq0$ and $E=(a-1)D_3+bD_4$. Then $D$ is big, and the configuration $(D,E)$ has connected sections by Lemma \ref{lemma7.2}. Applying Theorem \ref{Theorem5.14}, for any curve $C$ not contained in the toric boundary on a very general surface $S$ in $|D|$ we have, 
\begin{align}\label{pic22eq}
    2g-2 \geq C\cdot((a-4)D_3+(b+l_1+l_2-2)D_4.
\end{align}
By Theorem \ref{Theorem5.17} , if $a\geq3$ and $b\geq 2-l_1-l_2$, then the natural restriction map Pic$(X_{\Sigma}) \rightarrow$ Pic$(S)$ is an isomorphism.
Thus, any curve $C$ is rationally equivalent to the complete intersection of $S$ with a divisor in the class $cD_2+dD_4$. If $C$ is not contained in the
boundary, an intersection number calculation yields
\begin{equation*}
     2g-2\geq c(l_1a(a-4)+a(b+l_1+l_2-2)+b(a-4))+da(a-4).
\end{equation*}
The degree of such a curve $C$ with respect to ample class $H=D_3+D_4$ is given by
\begin{equation*}
    \text{deg}(C)=  c(l_1a+b+a)+da.
\end{equation*}
Let $\epsilon_0=\dfrac{1}{l_1a+b+a}$. Then we have
\begin{equation*}
    2g-2\geq\epsilon_0\cdot \text{deg}(C).
\end{equation*}
By combining Lemma \ref{lemma7.1} and Corollary \ref{corollary111}, we have the results.\\\\


\textbf{Case 3.0.1.}
    Let $r\geq3,a\geq3,b\geq1$. Then choose $D=dD_1+eD_4+fD_6$ and $E=dD_1+eD_3+(f-1)D_6$ with $d\geq0,e\geq0,f\geq1$. Then $D$ is big, and the configuration $(D,E)$ has connected sections by Lemma \ref{lemma7.2}. For other values of $r,a,b$, we may have to modify the choice of $(D,E)$ as in Lemma \ref{lemma7.2}. But the remaining process is similar.
    Applying Theorem \ref{Theorem5.14}, for any curve $C$ not contained in the toric boundary on a very general surface $S$ in $|D|$ we have,
\begin{align}\label{eq555}
      2g-2 \geq C \cdot ((d+a+r-2)D_1+(e+b-2)D_4+(f-3)D_6).
\end{align}
By Theorem \ref{Theorem5.17}, if $d\geq0$, $e\geq1$ and $f\geq 2$, then the natural restriction map Pic$(X_{\Sigma}) \rightarrow$ Pic$(S)$ is an isomorphism. Thus, any curve $C$ is rationally equivalent to the complete intersection of $S$ with a divisor in the class $jD_1+kD_3+lD_5$. If $C$ is not contained in the
boundary, an intersection number calculation yields

\begin{align*}
    2g-2\geq
    &j((e+b-2)f+(f-3)(e+bf)) +\\	&k((d+r+a-2)f+(d+af)(f-3)) +\\	&l((d+r+a-2)e+(e+b-2)(d+er)).
\end{align*}

The degree of such a curve $C$ with respect to ample class $H=D_1+D_4+D_6$ is given by
\begin{equation*}
    \text{deg}(C)=  j(f+e+bf)+k(f+d+af)+l(e+d+er).
\end{equation*}
If $d\geq0,e\geq2,f\geq3$, then choose 
\[\epsilon_0=\dfrac{1}{f+d+e+bf+af+er)}.\]
Then we have
\begin{align*}
    2g-2\geq
    jf+4kf+l(4e+6)\geq \epsilon_0.~\text{deg}~C.
\end{align*}
By combining Lemma \ref{lemma7.1} and Corollary \ref{corollary111}, we have the results. \\\\
 
\textbf{Case 3.0.2.}
    Let 
     $D=dD_1+(e-bf)D_4+fD_6$ and $E=(d-1)D_1+(e-bf+b)D_4+(f-1)D_6$ with $d\geq,e\geq0,f\geq1$. Then $D$ is big, and the configuration $(D,E)$ has connected sections by Lemma \ref{lemma7.2}. For any curve $C$ that not contained in the toric boundary on a very general surface $S$ in $|D|$ we have,
\begin{align*}
    2g-2\geq C.((d+r+a-3)D_1+(e+2b-bf-2)D_4+(f-3)D_6).
\end{align*}
Let $H=D_1+(1-b)D_4+D_6$. If $d\geq4-r-a$, $e-bf\geq3-3b$ and $f\geq 4$, then choose $\epsilon_0=1$. Then we have
\begin{align*}
    2g-2\geq \epsilon_0(C\cdot H).
\end{align*}
By combining Lemma \ref{lemma7.1} and Corollary \ref{corollary111}, we have the results. \\\\

\textbf{Case 3.1.1.~~}
Let $D=dD_{v_1}+fD_{u_1}+(e+f)D_{z_1}$ and $E=(d-1)D_{v_1}+fD_{u_1}+(e+f-1)D_{z_1}+(f-1)D_6$ with $d\geq1,e\geq1,f\geq0$.
Then $D$ is big, and the configuration $(D,E)$ has connected sections by Lemma \ref{lemma7.2}.  Applying Theorem \ref{Theorem5.14}, for any curve $C$ not contained in the toric boundary on a very general surface $S$ in $|D|$ we have, 
\begin{align}
    2g-2 \geq C\cdot ((d+b_1-3)D_{v_1}+(f-1)D_{u_1}+(e+f-3)D_{z_1}).
\end{align}
Let $H=D_{v_1}+D_{u_1}+2D_{z_1}$. If $d\geq4,f\geq2,e\geq3$, then we have
\begin{align}\label{equation31}
    2g-2 \geq C\cdot H.
\end{align}
Applying (\ref{equation31}) and Lemma \ref{lemma7.1} on Corollary \ref{corollary111}, we have the results. \\\\


\textbf{Case 3.1.2.~~}
Let $D=dD_{v_1}+fD_{u_1}+(e+f)D_{z_1}$ and $E=(d-1)D_{v_1}+fD_{u_1}+(e+f-1)D_{z_1}+(f-1)D_6$ with $d\geq1,e\geq1,f\geq0$.
Then $D$ is big, and the configuration $(D,E)$ has connected sections by Lemma \ref{lemma7.2}.
Applying Theorem \ref{Theorem5.14}, for any curve $C$ not contained in the toric boundary on a very general surface $S$ in $|D|$ we have, 
\begin{align}
    2g-2 \geq C\cdot ((d+b_1-3)D_{v_1}+fD_{u_1}+(e+f-3)D_{z_1}).
\end{align}
Let $H=D_{v_1}+D_{u_1}+2D_{z_1}$.
If $d\geq4,e\geq4,f\geq1$, then we have
\begin{align}\label{equation32}
    2g-2 \geq C\cdot H.
\end{align}
Applying (\ref{equation32}) and Lemma \ref{lemma7.1} on Corollary \ref{corollary111}, we have the results.\\\\

\textbf{Case 3.1.3.~~}
Let $D=dD_{v_1}+fD_{u_1}+(e+f)D_{z_1}$ and $E=(d-1)D_{v_1}+fD_{u_1}+(e+f-1)D_{z_1}+(f-1)D_6$ with $d\geq1,e\geq1,f\geq0$.
Then $D$ is big, and the configuration $(D,E)$ has connected sections by Lemma \ref{lemma7.2}.
Applying Theorem \ref{Theorem5.14}, for any curve $C$ not contained in the toric boundary on a very general surface $S$ in $|D|$ we have, 
\begin{align}
    2g-2 \geq C\cdot ((d+b_1+c_2-2)D_{v_1}+(f-1)D_{u_1}+(e+f-4)D_{z_1}).
\end{align}
Let $H=D_{v_1}+D_{u_1}+2D_{z_1}$. If $d\geq3-b_1-c_2,e\geq4,f\geq2$, then we have
\begin{align}\label{equation33}
    2g-2 \geq C\cdot H.
\end{align}
Applying (\ref{equation33}) and Lemma \ref{lemma7.1} on Corollary \ref{corollary111}, we have the results.\\\\

\textbf{Case 3.1.4.~~}
Let $D=dD_{v_1}+fD_{u_1}+(e+f)D_{z_1}$ and $E=(d-1)D_{v_1}+fD_{u_1}+(e+f-1)D_{z_1}+(f-1)D_6$ with $d\geq1,e\geq1,f\geq0$.
Then $D$ is big, and the configuration $(D,E)$ has connected sections by Lemma \ref{lemma7.2}.
Applying Theorem \ref{Theorem5.14}, for any curve $C$ not contained in the toric boundary on a very general surface $S$ in $|D|$ we have, 
\begin{align}
    2g-2 \geq C\cdot ((d+b_1+b_2)D_{v_1}+(f-3)D_{u_1}+(e+f-4)D_{z_1}.
\end{align}
Let $H=D_{v_1}+D_{u_1}+2D_{z_1}$. If $d\geq1-b_1-b_2,e\geq2,f\geq4$, then we have
\begin{align}\label{equation34}
    2g-2 \geq C\cdot H.
\end{align}
Applying (\ref{equation34}) and Lemma \ref{lemma7.1} on Corollary \ref{corollary111}, we have the results.\\\\

\textbf{Case 3.1.5.~~~}
Let $D=dD_{v_1}+fD_{u_1}+(e+f)D_{z_1}$ and $E=dD_{v_1}+(f-1)D_{u_1}+(e+f-1)D_{z_1}$ with $d\geq1,e\geq1$ and $f\geq0$.
Then $D$ is big, and it can be shown that the configuration $(D,E)$ has connected sections. 
Applying Theorem \ref{Theorem5.14}, for any curve $C$ not contained in the toric boundary on a very general surface $S$ in $|D|$ we have, 
\begin{align}
    2g-2 \geq C\cdot ((d+b_1-1)D_{v_1}+(f-3)D_{u_1}+(e+f-3)D_{z_1}.
\end{align}
Let $H=D_{v_1}+D_{u_1}+2D_{z_1}$. If $d\geq2,f\geq4,e+f>=5$, then we have
\begin{equation}\label{eq38}
    2g-2 \geq C\cdot H.
\end{equation}
Applying (\ref{eq38}) and Lemma \ref{lemma7.1} on Corollary \ref{corollary111} we have the results.

\end{proof}

\printbibliography
\end{document}